\documentclass[10 pt]{amsart}
\usepackage{amsmath,amsthm,amssymb,mathtools}
\usepackage{xcolor}
\usepackage[shortlabels]{enumitem}
\definecolor{darkspringgreen}{rgb}{0.09, 0.45, 0.27}
\usepackage[colorlinks=true,citecolor=darkspringgreen,linkcolor=darkspringgreen]{hyperref}
\usepackage[capitalize]{cleveref}
\usepackage{autonum}
\usepackage{subcaption}
\usepackage{bbm}
\usepackage{tikz}
\usepackage{mathabx} 

\usetikzlibrary{intersections}

\newcommand{\dd}{\mathrm{d}}
\newcommand{\bP}{\mathbb{P}}
\newcommand{\cP}{\mathcal{P}}

\newcommand{\N}{\mathbb{N}}
\newcommand{\R}{\mathbb{R}}
\newcommand{\C}{\mathbb{C}}
\newcommand{\E}{\mathbb{E}}
\newcommand{\eps}{\varepsilon}

\allowdisplaybreaks

\DeclareMathOperator{\spec}{\mathrm{spec}}

\DeclareMathOperator*{\esup}{ess\,sup}

\DeclarePairedDelimiter\abs{|}{|}
\DeclarePairedDelimiterX{\inner}[2]{\langle}{\rangle}{#1,#2}
\DeclarePairedDelimiterX{\norm}[1]{\|}{\|}{#1}

\numberwithin{equation}{section}

\newtheorem{thm}{Theorem}[section]

\newtheorem{prop}[thm]{Proposition}
\newtheorem{cor}[thm]{Corollary}
\newtheorem{lemma}[thm]{Lemma}
\theoremstyle{definition}

\newtheorem{ass}[thm]{Assumption}

\newtheorem{ex}[thm]{Example}
\theoremstyle{remark}
\newtheorem{rmk}[thm]{Remark}

\Crefname{cor}{Corollary}{Corollaries}

\title{Separation of time scales in weakly interacting diffusions}
\author[Adams]{Zachary P. Adams}
\address[1]{Freie Universit\"at Berlin \&~Scads.AI Leipzig}
\email{zachary.adams@fu-berlin.de}
\author[Engel]{Maximilian Engel}
\address[2]{University of Amsterdam \& Freie Universit\"at Berlin}
\email{maximilian.engel@fu-berlin.de}
\author[Gvalani]{Rishabh S. Gvalani}
\address[3]{D-MATH, ETH Z\"urich, R\"amistraße 101, 8001 Z\"urich}
\email{rgvalani@ethz.ch}
\date{}

\begin{document}
\begin{abstract}
We study metastable behaviour in systems of weakly interacting Brownian particles with localised, attractive potentials which are smooth and globally bounded. 
In this particular setting, numerical evidence suggests that the particles converge on a short time scale to a ``droplet state'' which is \emph{metastable}, i.e.~persists on a much longer time scale than the time scale of convergence, before eventually diffusing to $0$.  

In this article, we provide rigorous evidence and a quantitative characterisation of this separation of time scales.  Working at the level of the empirical measure, we show that (after quotienting out the motion of the centre of mass) the rate of convergence to the quasi-stationary distribution, which corresponds with the droplet state,  is $O(1)$ as the inverse temperature $\beta \to \infty$. Meanwhile the rate of leakage away from its centre of mass is $O(e^{-\beta})$. Futhermore, the quasi-stationary distribution is localised on a length scale of order $O(\beta^{-\frac12})$. We thus provide a partial answer to a question posed by Carrillo, Craig, and Yao (see~\cite[Section 3.2.2]{CCY19}) in the microscopic setting.

\end{abstract}
\maketitle

\section{Introduction}
\label{sec:introduction}

Consider the following system of weakly interacting diffusions
\begin{equation}
    \dd X_t^i = -\frac{1}{N}\sum_{i=1}^N \nabla W (X_t^i -X_t^j) \, \dd t + \sqrt{2 \beta^{-1}} \dd
 B_t^i \, ,
    \label{eq:particlesystem}
\end{equation}
where $(X_t^i)_{i=1}^N$, $N \in \mathbb{N}$, represent the positions of $N$ exchangeable particles (or agents or spins), taking values in some smooth manifold $\Omega$, $W$ is a smooth even interaction potential with well-behaved growth properties, and $(B_t^i)_{i=1}^N$ are $N$ independent Brownian motions.

Depending on the choice of the interaction potential $W$, the system in \eqref{eq:particlesystem} can exhibit a wide variety of interesting dynamical phenomena. We are particularly interested in the case in which $W$ is attractive (i.e.~$\nabla W (x)\cdot x \geq 0$), globally bounded, and has smooth and globally bounded derivatives of all orders. In this situation the pairwise attraction mediated by $W$ competes with the diffusive behaviour produced by the independent Brownian motions.  This competition is most easily observed  at the level of the thermodynamic limit $N\to \infty$ when $\Omega$ is bounded. Consider the empirical measure $\mu^N_t$ of the system~\eqref{eq:particlesystem}, which is defined as 
\begin{align}
\mu^N_t \coloneqq \frac1N \sum_{i=1}^N\delta_{X_t^i} \, .
\end{align}
Then, assuming $W$ is sufficiently well-behaved, it is well-known (see, for example, \cite{Szn91}) that we have the estimate 
\begin{equation}
\sup_{t \in [0,T]} \E \left[ d_2^2(\mu_t,\mu_t^N)  \right] \lesssim \frac1N \, ,
\end{equation}
for any $T<\infty$, where $d_2(\cdot,\cdot)$ is the $2$-Wasserstein distance and $\mu_t$ is the unique distributional solution of the  nonlocal parabolic PDE
\begin{equation}
\partial_t \mu_t = \beta^{-1}\Delta \mu_t + \nabla \cdot(\mu_t\nabla W * \mu_t) \, ,
\label{eq:MVintro}
\end{equation}
often referred to as the McKean--Vlasov equation. The above PDE can exhibit the phenomenon of phase transitions (see \cite{CGPS20, CP10}): for $\beta$ sufficiently small, equation~\eqref{eq:MVintro} has a unique steady state (in the space of probability measures) while for $\beta$ sufficiently large, equation~\eqref{eq:MVintro} possesses multiple steady states. This switch between one and multiple steady states for \eqref{eq:MVintro} is caused exactly by the aforementioned competition between the diffusive behaviour produced by the Laplacian, which dominates in the $\beta \ll 1$ regime and the tendency to aggregate due to the nonlocal drift, which prevails in the $\beta \gg 1$ regime. 

\subsection*{The setting: bounded potential on unbounded domain}

In this article, we are interested in the setting of $\Omega=\R^d$. Since $W$ and all its derivatives are assumed to be bounded, solutions of \eqref{eq:MVintro} are not uniformly tight and the equation has no steady states which are probability measures. A proof of this can be found in \cite[Theorem 3.1]{CDP19}. 

Even though \eqref{eq:MVintro} has no steady states, the competition between the diffusive and attractive terms manifests itself in a different manner. Numerical experiments (see \cite[Figure 8]{CCY19}) suggest that if the potential $W$ is sufficiently localised, $\beta \gg 1$, and the initial datum is well-prepared, then solutions of \eqref{eq:MVintro} appear to converge to a \emph{localised droplet state} and stay close to it for a very long time before eventually converging to $0$, as expected.  Thus, there is a separation of time scales between convergence to the droplet state and the eventual escape of mass to infinity and convergence to $0$. This separation is an instance of what is commonly referred to as \emph{dynamical metastability}, wherein solutions of a dissipative dynamical system converges on a fast time scale to a submanifold of its state space, along which the time evolution is slow. 
We refer the reader to \cref{fig:schematic}, where we provide a simple schematic sketch of this phenomenon, which has been studied in the context of other equations as well, such as reaction-diffusion equations (see \cite{CP90,AM25}) or the Cahn--Hilliard equation (see \cite{OR07}).  We also point the reader to~\cite[Section 3.2.2]{CCY19} for a more detailed discussion of this phenomenon in the context of~\eqref{eq:MVintro}.

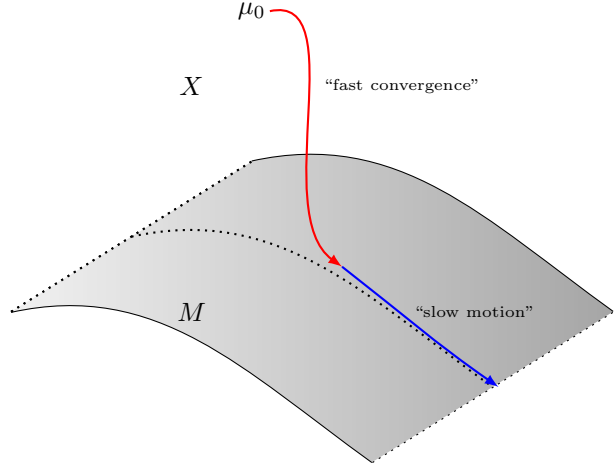
\begin{figure}
\begin{tikzpicture}[xscale=0.8, >=latex]
\path[draw, thick, name path=border1] (0,0) to [out=10,in=150] (6,-2);
\path[draw, thick, name path=border1,dotted] (2,-4) to  (6,-2);
\path[draw, thick, name path=border1] (-4,-2) to [out=10,in=150] (2,-4);
\path[draw, thick, name path=border1,dotted] (-4,-2) to  (0,0);
\shade[left color = gray!10, right color=gray!70]
(0,0) to [out=10, in =150] (6,-2) -- (2,-4) to[out=150,in=10] (-4,-2) -- cycle;
\node at (-1,-2) {$M$};
\node at (-1,1) {$X$};
\path[draw, thick, name path=border1,dotted] (-4,-2) to  (0,0);
\path[draw, thick, dotted] (-2,-1) to [out=10,in=150] (4,-3);
\node[circle] at (0,2)  {$\mu_0$};
\path[draw, thick, color=red,->] (0.3,2)  to [out=10,in=150] (1.5,-1.4) ;
\path[draw, thick, color=blue,->] (1.5,-1.4) to[out=-32,in=151]  (4.1,-3);
\node[circle] at (2.5,1)  {\tiny ``fast convergence''};
\node[circle] at (3.7,-2)  {\tiny ``slow motion''};
\end{tikzpicture}
\caption{A schematic depiction of dynamical metastability: the coloured paths represent trajectories of a dynamical system which can be divided into a regime of fast convergence towards the slow submanifold $M$ (red), and slow motion along $M$ (blue).}
\label{fig:schematic}
\end{figure}

\subsubsection*{Absence of invariant probability measure for the particle system}
We are studying this separation of time scales as it shows up in the corresponding particle system. In analogy to the absence of steady states for the thermodynamic limit, the particle system \eqref{eq:particlesystem} does not have an invariant probability measure under our assumptions. To see this, consider the Fokker--Planck equation associated to the law $\rho_t(x) \, \mathrm{d}x $ of \eqref{eq:particlesystem}, 
\begin{equation}
\partial_t \rho_t =\beta^{-1}\Delta \rho_t + \nabla \cdot (\rho_t\nabla H_N )  \, ,
\end{equation}
where  the Hamiltonian $H_N : (\R^d)^N \to \R$ is given by 
\begin{equation}
H_N(x) \coloneqq \frac{1}{2N}\sum_{i,j=1}^N W(x_i -x_j) \, .
\end{equation}
We define $f_t \coloneqq\rho_t e^{\beta H_N}$ and observe that it solves the following equation
\begin{align}
\partial_t f_t = & \, \beta^{-1}\Delta f_t - \nabla H_N \cdot \nabla f_t \\
=& \, \beta^{-1}e^{\beta H_N}\nabla \cdot (e^{-\beta H_N}\nabla f_t) \, .
\end{align}
Multiplying by $f_t$ and integrating by parts against the measure $e^{-\beta H_N(x)} \, \mathrm{d}x$, we obtain the estimate
\begin{equation}
\frac{\dd }{\dd t} \int_{(\R^d)^N} |f_t|^2 e^{-\beta H_N(x)}\, \mathrm{d}x =  -2 \beta^{-1}\int_{(\R^d)^N} |\nabla f_t|^2 e^{-\beta H_N(x)}\, \mathrm{d}x \, .
\label{eq:energyestimate} 
\end{equation}
We now use the fact that $W$ is globally bounded by some constant $K<\infty$ along with the Nash inequality \cite{Nash1958} to assert that there exists a constant $C=C(N,d, K,\beta)$ such that
\begin{equation}
\frac{\dd }{\dd t} \int_{(\R^d)^N} |f_t|^2 e^{-\beta H_N(x)}\, \mathrm{d}x \leq -C \left[\int_{(\R^d)^N} |f_t|^2 e^{-\beta H_N(x)}\, \mathrm{d}x\right]^{\frac{Nd+2}{Nd}} \, .
\end{equation}
Applying Gr\"onwall's inequality, rewriting in terms of $\rho_t$,  and using the global boundedness of $W$ again, we see that 
\begin{equation}
\label{eq:polynomialdecay}
\int_{(\R^d)^N} |\rho_t|^2 \, \mathrm{d}x\lesssim (1+t)^{-\frac{Nd}{2}} \,,
\end{equation} 
which contradicts the existence of an invariant probability measure. 

\subsubsection*{Phenomenology of the particle system}
In the absence of an invariant probability measure, one observes a similar dynamical picture as that seen for the thermodynamic limit: the particles almost instantaneously converge to a droplet state in which they are all close to their collective centre of mass; a configuration which is analogous to the droplet state observed in the thermodynamic limit. This configuration persists for a long time before eventually dissipating, similar to the slow motion regime observed in the $N \to \infty$ limit. We refer the reader to~\cref{fig:2dsim} where we have demonstrated this phenomenon for a particular choice of $W$ in dimension $d=2$. 
\begin{figure}
\centering
\begin{minipage}{0.3\linewidth}
\centering
\includegraphics[scale=0.28]{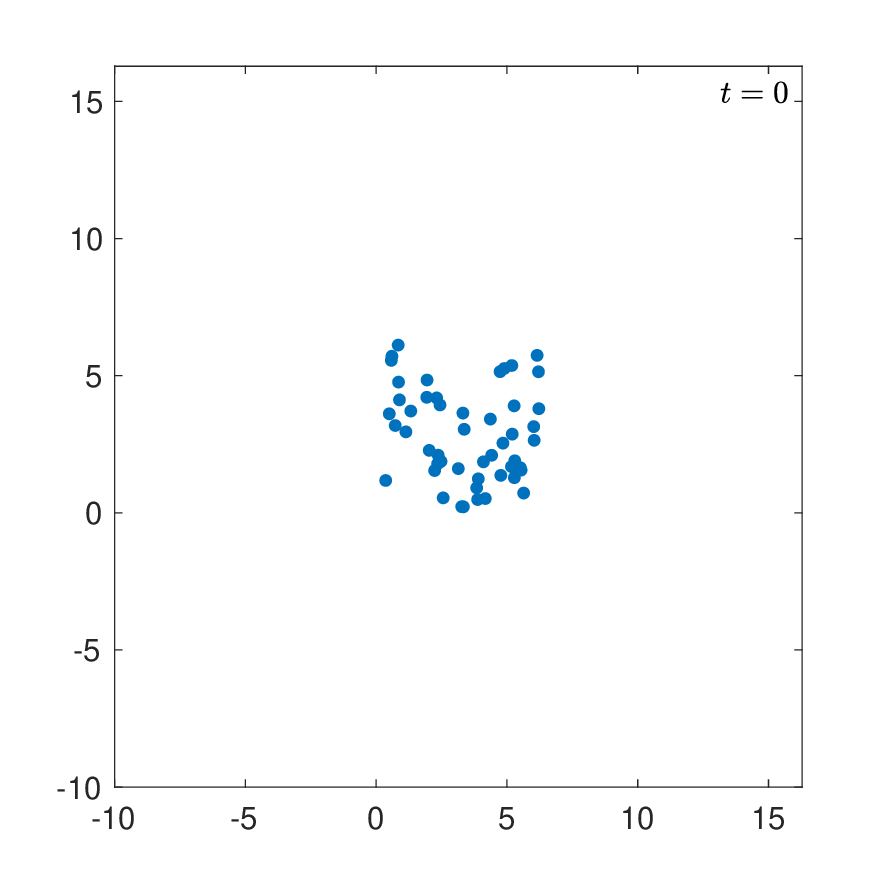}
\end{minipage}
\begin{minipage}{0.3\linewidth}
\centering
\includegraphics[scale=0.28]{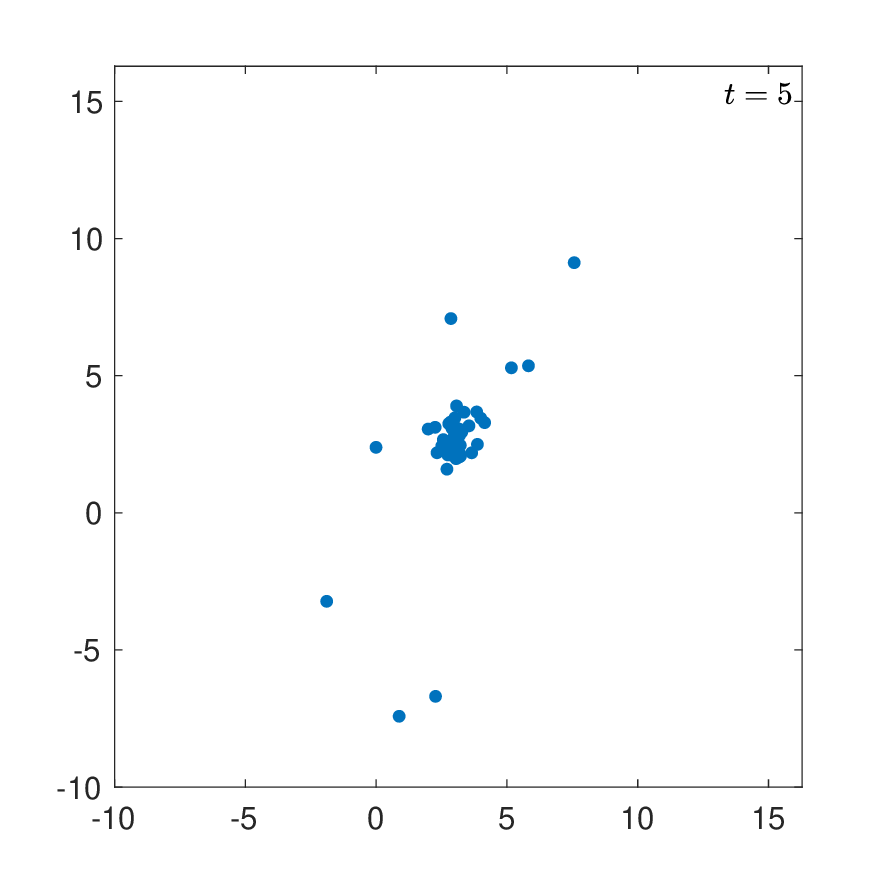}
\end{minipage}
\begin{minipage}{0.3\linewidth}
\centering
\includegraphics[scale=0.28]{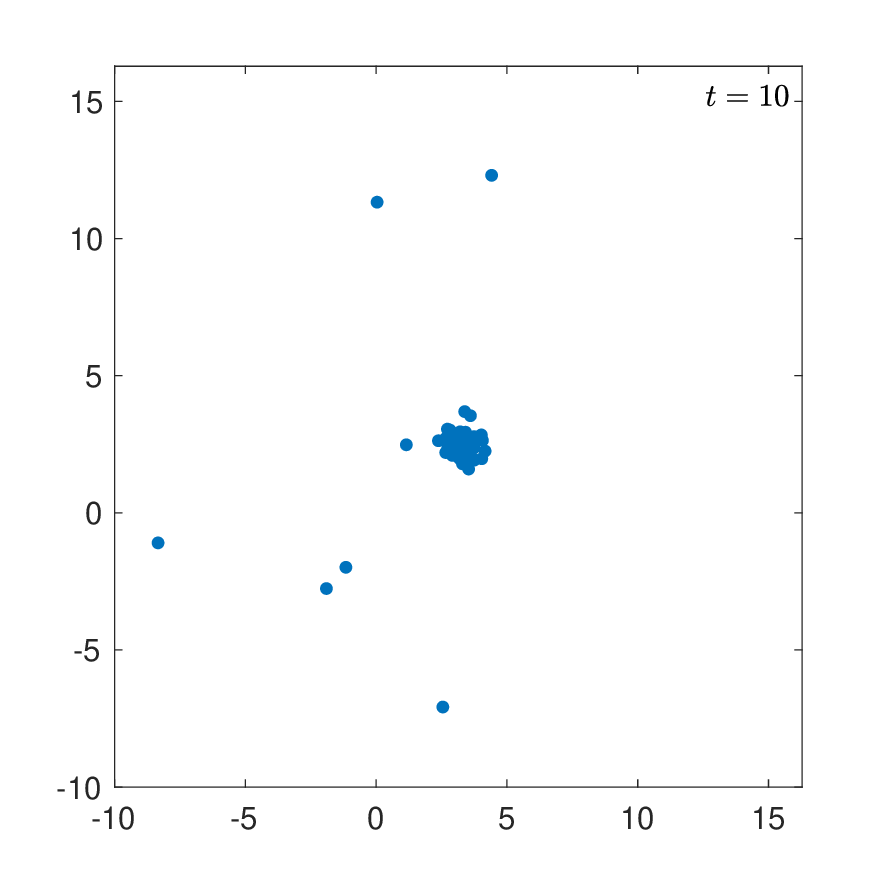}
\end{minipage}
\begin{minipage}{0.3\linewidth}
\centering
\includegraphics[scale=0.28]{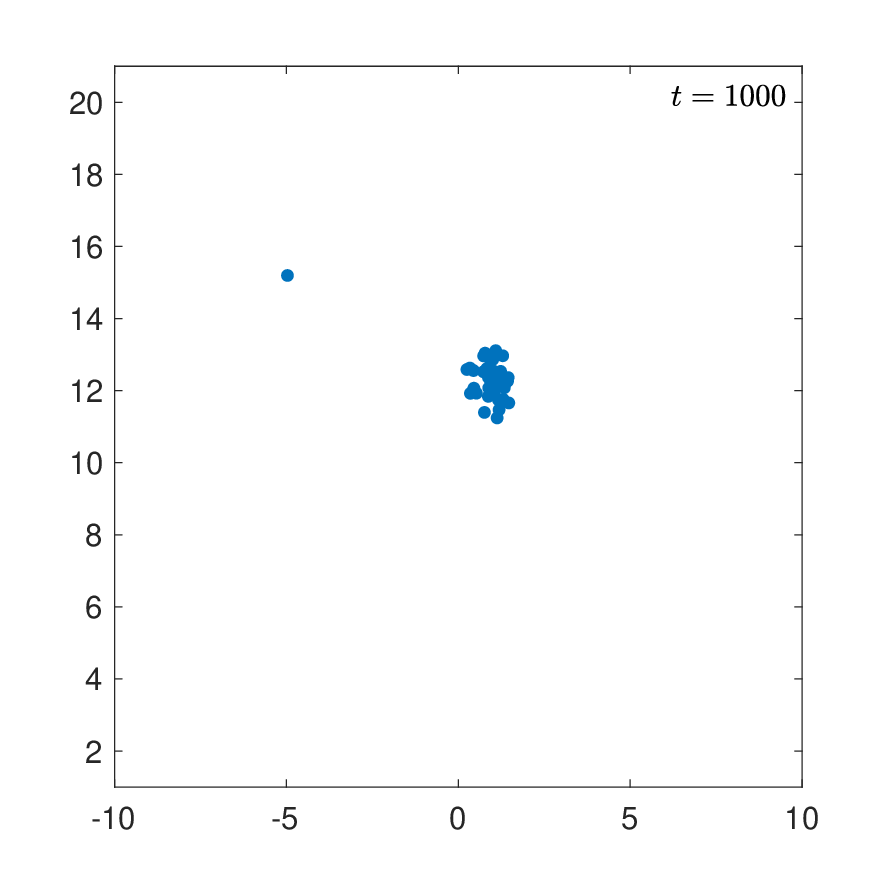}
\end{minipage}
\begin{minipage}{0.3\linewidth}
\centering
\includegraphics[scale=0.28]{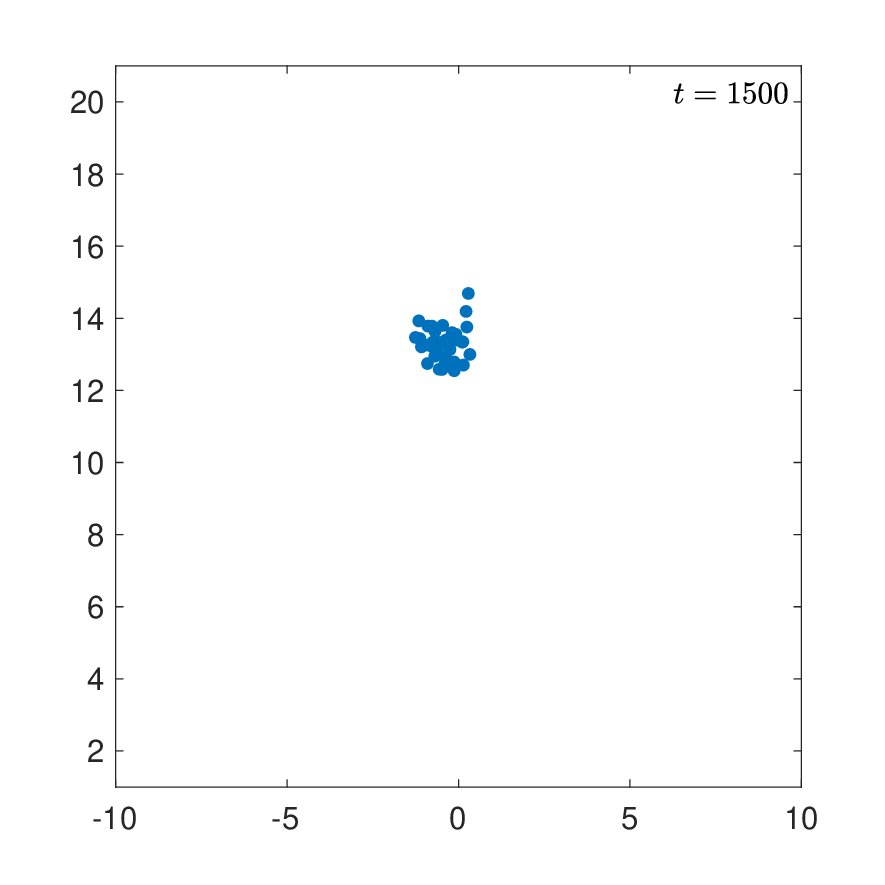}
\end{minipage}
\begin{minipage}{0.3\linewidth}
\centering
\includegraphics[scale=0.28]{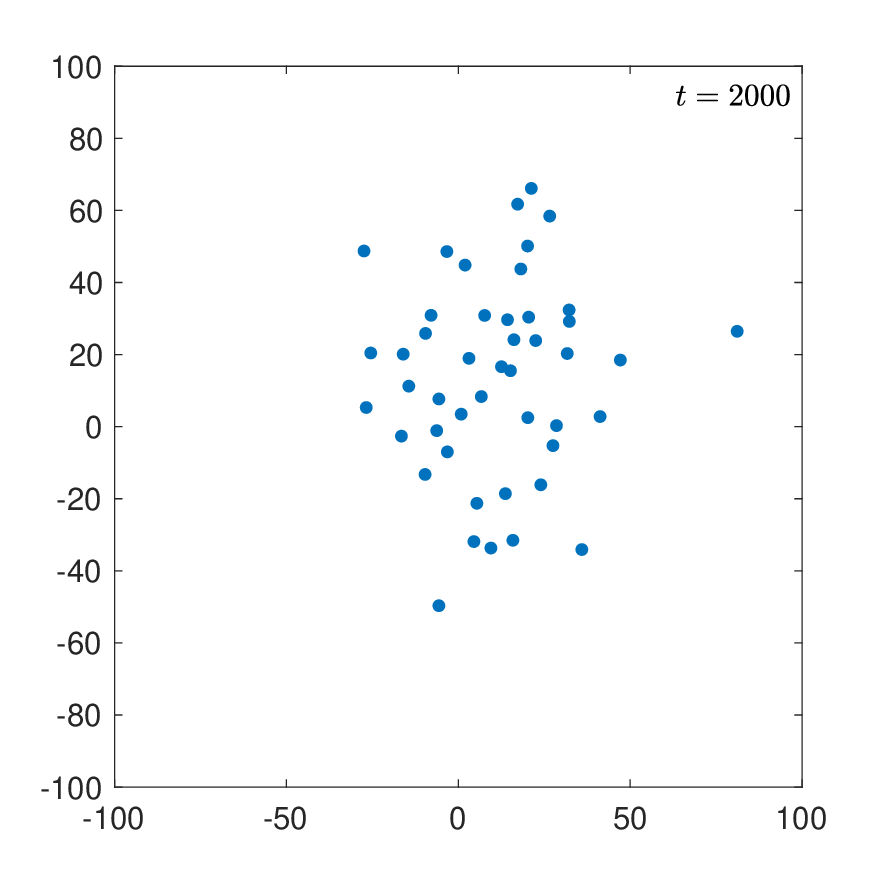}
\end{minipage}
\caption{
Simulation showing convergence to a single droplet state on a short time scale, followed by the eventual escape of particles from the droplet on a much longer time scale. 
The simulations are performed with $W=1-e^{-|x|^2}$ and $\beta=0.3$.}
\label{fig:2dsim}
\end{figure}

A key difference between the particle system and the limit is that the centre of mass of the particles is not conserved while that of the limit $\mu_t$ is, i.e.~$\int_{\R^d} x \mu_t(x) \, \mathrm{d}x = \int_{\R^d} x \mu_0(x) \, \mathrm{d}x$ for all times $t\geq0$. In contrast, one can check that the centre of mass of the particles satisfies
\begin{align}
\bar{X}_t^N \coloneqq & \,\int_{\R^d} x \, \mathrm{d}\mu_t^N \\
=& \, \frac1N\sum_{i=1}^N X_t^i = \frac{\sqrt{2\beta^{-1}}}{N}\sum_{i=1}^N B_t^i  \, ,
\end{align}
where we have used the fact that $W$ is even. Thus, $\bar{X}_t^N$ performs Brownian motion with variance of order $N^{-1}$, and this  motion vanishes in the limit. However, since we are only interested in the distance of the particles from their centre of mass and not the motion of the centre of mass itself, we quotient out this random drift. This quotienting procedure is explained in more detail in \cref{sec:prelim}. 

Based on these observations, we study the separation of time scales for the particle system \eqref{eq:particlesystem}~instead of the PDE \eqref{eq:MVintro}. 
The reason for this is two-fold. 
Firstly,  our approach to the problem involves studying the spectrum of the generator of the interacting particle system~\eqref{eq:particlesystem}. 
As this generator is linear, it can be tackled using standard tools from spectral theory (as opposed to the limit~\eqref{eq:MVintro}, which is nonlinear). 
Secondly, if one views the interacting particle system~\eqref{eq:particlesystem} as representing the ``ground truth'' and the limit~\eqref{eq:MVintro} only as an approximation, studying the time scale separation on this level is more appropriate.\\

\subsection*{Main result}
We now present a heuristic version of the main result of this paper. Let $\mu_t^N$ denote the empirical measure of the system of SDEs \eqref{eq:particlesystem} (after quotienting out the centre of mass $\bar{X}_t^N$), and let $\bar{\mu}_N$ denote the empirical measure of the droplet configuration described above (also after quotienting out $\bar{X}^N_t$). 
Assume that $W$ satisfies the previous boundedness assumptions and is attractive.
Furthermore, assume that $\beta \gg 1$, and that the particles are sufficiently localised (in a sense that we make more precise later on). 

\begin{thm}
 For any $N \in \mathbb N$ and large $\beta\gg1$, there exist constants $\lambda_{1,\beta,N}>0$ and $\lambda_{2,\beta,N}>0$  such that $\mu_t^N$ converges to the random droplet configuration $\bar{\mu}^N$ on a time scale of order $\lambda_{2,\beta,N}^{-1}$, 
while the leakage from the droplet state occurs on a time scale of order $\lambda_{1,\beta,N}^{-1}$, with
\begin{align}
\lim_{\beta \to + \infty} \lambda_{1,\beta,N}=0, \qquad \lim_{\beta \to + \infty} \lambda_{2,\beta,N}= \lambda_{2,N} >0 \,. 
\end{align}
Moreover, the limit $\lambda_{2,N}$ is bounded away from $0$ uniformly in $N$. 
Finally, $\bar{\mu}^N$ is localised on a length scale of order $\beta^{-\frac12}$, i.e. 
\begin{equation}
\mathbb{E}\left[\int_{\R^d}(x-\int \bar{x}\,\mathrm{d}\bar{\mu}^N(\bar{x}))^2\,\mathrm{d}\bar{\mu}^N(x) \right] \leq \beta^{-1} \, . 
\end{equation} 
\label{thm:heuristic}
\end{thm}

The mathematically precise version of the above result can be found in~\cref{sub:statement_of_the_main_results}. The proof relies on studying the system \eqref{eq:particlesystem} on a bounded domain with killing and using the theory of quasi-stationary distributions (QSDs) (introductions to which can be found in \cite{CMSM13}~or \cite{CV16,CV17a}). The quantities $\lambda_{1,\beta},\lambda_{2,\beta}$ then correspond to the top two eigenvalues of the generator of the killed process. Furthermore, the droplet configuration mentioned in~\cref{thm:heuristic} is exactly the empirical measure of points which are sampled from the quasi-stationary distribution of the system.

\subsection*{Structure of the paper}
Section 2 gives an overview of the notation, techniques, and results of this paper. In Section 2.1, we specify the assumptions on the interaction potential $W$ and, in Section 2.2, we introduce the setting of the projected dynamics after quotienting out the motion of the centre of mass.
After introducing a more precise definition of the droplet state $\bar{\mu}^N$ in Section 2.3, we introduce the (sub-)Markovian semigroups and their generators used in our setting in Section 2.4. The sub-Markovian part concerns the situation where trajectories hitting the boundary of the domain of interest are disregarded, and conditioning on survival within this domain is considered. This leads to the notion of a quasi-stationary distribution in Section 2.5 that helps to describe the droplet state and convergence to it.
In Section 2.6 our main results are stated: Theorem 2.2 describes the rate of escape from the droplet state as exponentially small in the inverse temperature $\beta$ as $\beta \to \infty$, while Theorem 2.5 yields a precise description of the $\beta \to \infty$-convergence for higher order eigenvalues; in particular, this yields an $\mathcal O (1)$ rate for the convergence to the QSD and the associated droplet state. 
Theorem 2.6 then translates these spectral results into a precise characterisation of the droplet state and convergence to it, along with leakage from the droplet state, occurring on different time scales. In the following, Section 3 contains the proof of Theorem 2.2, Section 4 gives the proof of Theorem 2.5 and Section 5 is dedicated to proving Theorem 2.6.

\section{Notation, preliminaries and main results}\label{sec:prelim}
We start this section by introducing notation needed for the rest of the paper. Given a Polish space $(X,d)$, we denote by $\cP(X)$ the space of Borel probability measures on $X$ and by $\cP_p(X)\subset \cP(X)$, $1 \leq p < \infty$, the set of those measures with finite $p$th moments. 
We also denote by $\mathcal{M}_+(X)$ the space of non-negative, $\sigma$-finite measures on $X$. We will often need to equip these spaces with some topology. For $ \cP(X)$, we work either with the weak topology, i.e.~the coarsest topology such that the map $\cP(X) \ni \mu \mapsto \int f \, \mathrm{d}{\mu}$ is continuous for every bounded and continuous $f$, or with the topology induced by the total variation metric $\lVert \cdot\rVert_{\mathrm{TV}}$, 
\begin{equation}
\lVert \mu -\nu \rVert_{\mathrm{TV}} \coloneqq \sup_{A \in \mathcal{B}(X)} |\mu(A) -\nu(A)| \, ,
\end{equation}
where $\mathcal{B}(X)$ is the Borel sigma-algebra on $X$. We equip the spaces $\cP_p(X)$ with the $p$-Wasserstein transportation cost metric $W_p(\cdot,\cdot)$, defined as follows
\begin{equation}
W_p(\mu,\nu) \coloneqq   \left(\inf_{\pi \in \Pi(\mu,\nu)} \int_{X\times X} |x-y|^p\, \mathrm{d}\pi(x,y)\right)^{\frac 1p} \,,
\end{equation}
where $\Pi(\mu,\nu)$ denotes the set of all couplings between the measures $\mu$ and $\nu$. We also need the following dual formulation of $W_1(\cdot,\cdot)$:
\begin{equation}
W_1(\mu,\nu) = \sup_{f \in \mathrm{Lip}_1(X)}\int_X f\, \mathrm{d}(\mu-\nu) \, ,
\label{eq:W1}
\end{equation}
where
\begin{equation}
\mathrm{Lip}_1(X)\coloneqq\{f\in C(X): |f(x)-f(y)|\leq d(x,y)\quad \forall x,y \in X\}\, .
\end{equation}

Given an open set $U \subseteq \R^d$ and a measure $\mu \in \cP(U)$, we denote by $L^p(U;\mu)$, $1 \leq p \leq \infty$, the weighted Lebesgue spaces on $U$ and by $C^k(U)$ (and $C_c^k(U)$) the space of $k$-times continuously differentiable (and compactly supported) functions on $U$. 

\subsection{Assumptions on the potential \texorpdfstring{$W$}{W}}
It is clear from the discussion in the introduction that a result of the form we are searching for cannot possibly be true for all forms of $W$. The potential $W$ should be locally attractive at a minimum, to allow the formation of a droplet, but at the same time its gradient should decay near infinity, to allow particles to escape the droplet. The assumption below captures these restrictions on $W$.
\begin{ass}
The interaction potential $W$ is radial, with $W(x)\coloneqq w(|x|)$ for some $w:\R \to \R$ which is bounded, even and smooth with $w'$ being a Schwartz function. Finally, we assume that $w(0)=w'(0)=0$, $w''(0)> -\min_{\R} w''> 0$, and $w'(x) > 0$ for $x >0$. 
\label{ass:W}
\end{ass}

It follows from the above assumption that $W$ is $\lambda$-convex, i.e.
\begin{equation}
W((1-t)x +t y) \leq (1-t)W(x) + tW(y) - \frac\lambda2 t(1-t)|x-y|^2 
\label{eq:lconvex}
\end{equation}
for all $x,y,\in \R^d$ and $t\in [0,1]$, with 
\begin{equation}
\lambda = \inf_{x \in \R^d} \inf\left\{\kappa:\kappa \in  \spec( (D^2 W) (x))\right \} \, ,
\label{eq:lambdaconvexity}
\end{equation}
where $\spec(A)$ denotes the spectrum of the matrix $A$. Note that we necessarily have $\lambda<0$, since if $\lambda \geq 0$, we would have that $W$ is a bounded and convex function and thereby constant, which it is not.  The following form of $\lambda$-convexity is needed for our arguments
\begin{equation}
W\left(\frac{1}{N}\sum_{i=1}^N x_i\right) \leq \frac{1}{N}\sum_{i=1}^N W\left( x_i\right) +\frac{\lambda(1-N)}{2N^2} \sum_{i=1}^N|x_i|^2 \, , 
\label{eq:lconvexN}
\end{equation}
for any collection of points $(x_i)_{i=1,\dots,N}\subset \Omega$. The above inequality is a straightforward consequence of \eqref{eq:lconvex}.

\subsection{Projected dynamics}
We now describe the quotienting procedure alluded to in the introduction. In the following, we set $ \Omega =\R^d$. 
As discussed earlier, we are interested in the relative positions of the particles, and would like to disregard the motion of their centre of mass. To this end, we define the following linear subspaces
\begin{align}
\Gamma_N \coloneqq & \, \left\{x \in \Omega^N : x_i=c \in \Omega,\,\forall i =1,\dots,N   \right\},\\
\Gamma_N^\perp \coloneqq & \, \left\{x \in \Omega^N : \sum_{i=1}^N x_i =0    \right\}  \, .
\end{align}
Furthermore, we define $\Pi_N : \Omega^N \to \Gamma_N^\perp$ to be the orthogonal projection onto $\Gamma_N^\perp$. 
Let $X_t = (X_t^1, \dots, X_t^N)$ denote the strong solution to the system of SDEs in \eqref{eq:particlesystem}. Then, we define
\begin{equation}
Y_t \coloneqq \Pi_N \, X_t \, .
\end{equation}
Since $X_t$ is a solution of the SDE
\begin{equation}
\dd X_t = -\nabla H_N(X_t) \, \dd t + \sqrt{2\beta^{-1}} \dd B_t \, ,
\end{equation}
where $B_t = (B_t^1, \dots, B_t^N)$, it follows that $Y_t$ is a solution of
\begin{equation}
\dd Y_t = -\Pi_N \, \nabla H_N(X_t) \, \dd t + \sqrt{2\beta^{-1}} \dd (\Pi_N \, B_t) \, .
\end{equation}
We now note that, since $\Pi_N$ is orthogonal, $\bar {B_t}\coloneqq \Pi_N \, B_t $ is a standard $(N-1)d$-dimensional Brownian motion on $\Gamma_N^\perp$. Furthermore, since $(\nabla H_N)(x)=(\nabla H_N)(\Pi_N\, x)$ we have that $\Pi_N \, \nabla H_N(X_t) = \bar{\nabla}U_N (Y_t)$, where $\bar{\nabla}$ is the Euclidean gradient on $\Gamma_N^\perp$ and $U_N$ is the restriction of $H_N$ to $\Gamma_N^\perp$. To see this, note that for any $y \in \Gamma_N^\perp, v \in T_y \Gamma_N^\perp$
\begin{align}
\langle\bar{\nabla} U_N (y), v \rangle\,\
=& \, \left. \frac{d}{d\eps} U_N(y + \eps v) \right|_{\eps=0}\\
=& \, \left. \frac{d}{d\eps} H_N(y + \eps v) \right|_{\eps=0} \,
= \, \langle \nabla  H_N (y), v \rangle \, ,
\end{align}
where we have used the fact that the metric on $\Gamma_N^\perp$ is the restriction of the metric on $\R^{Nd}$. Using the fact that projections are idempotent and self-adjoint, we have that 
\begin{equation}
\langle\bar{\nabla} U_N (y), v \rangle = \langle \Pi_N \,\nabla  H_N (y), v \rangle \, .
\end{equation}
Thus, we have that $Y_t$ solves the SDE
\begin{equation}
dY_t = -\bar{\nabla} U_N(Y_t) \, dt + \sqrt{2\beta^{-1}} d \bar{B}_t \, ,
\label{eq:quotientparticle}
\end{equation}
implying that $Y_t$ is also a Markov process. Additionally, the positive measures ${\mu}_{N,X}\in \mathcal{M}_+(\Omega^N)$ and ${\mu}_{N,Y}\in \mathcal{M}_+(\Gamma_N^\perp)$, given by
\begin{equation}
\mu_{N,X}(\mathrm{d}x) \coloneqq \exp(-\beta H_N(x)) \, \mathrm{d}x \qquad \mu_{N,Y}(\mathrm{d}y) \coloneqq \exp(-\beta U_N(y)) \, \mathrm{d}y \, ,
\label{eq:muN}
\end{equation}
are invariant for the processes $X_t$ and $Y_t$, respectively. In the following sections, for the sake of notational convenience, we drop the bar on $\bar{\nabla}$ (resp.~$\bar{B}_t$) when referring to the gradient (resp.~Brownian motion) on $\Gamma_N^\perp$. Furthermore, since we will only be referring to and working with $\mu_{N,Y}$ in the remainder of the article, we will drop the subscript $Y$.

\subsection{The empirical measure} 
\label{sub:the_empirical_measure}
In this section, we provide a more precise mathematical formulation of what it means for the particles to droplet together. As mentioned previously, we are interested in the positions of particles relative to their centre of mass. Keeping this in mind, we arrive at the following natural notion of a droplet: we say the configuration of particles $x \in \Omega^N$  lies in a droplet of size $\ell>0$ if 
\begin{align} \label{eq:droplet_x}
\int_{\R^d}\left(z - \bar{x}\right)^2 \, \mathrm{d}\mu^{(N),x} \leq \ell^2 \, ,
\end{align}
where $\mu^{(N),x}$ is the empirical measure associated to the configuration $x$, i.e.
\begin{equation}
\mu^{(N),x} \coloneqq \frac1N\sum_{i=1}^N\delta_{x_i}\, 
\end{equation}
and
\begin{equation}
\bar{x}\coloneq\int_{\R^d} z \, \mathrm{d}\mu^{(N),x}\, .
\end{equation}
The above notion of droplet is agnostic to the choice of centre of mass, $\bar{x}$. Indeed, we can equivalently express condition~\eqref{eq:droplet_x} in terms of $\mu^{(N),y}\coloneqq  \mu^{(N), \Pi_N\, x}$, $y=\Pi_N \,x$, as 
\begin{align}  \label{eq:droplet_y}
\int_{\R^d}|z|^2 \, \mathrm{d}\mu^{(N),y} \leq \ell^2\, .
\end{align}
At the level of the configuration of particles, conditions~\eqref{eq:droplet_x} and~\eqref{eq:droplet_y} reduce to
\begin{equation}
\frac{1}{2N}\sum_{i,j=1}^N |x_i-x_j|^2 \leq \ell^2
\end{equation}
and, equivalently,
\begin{equation}
\frac{1}{N}\sum_{i=1}^N |y_i|^2 \leq \ell^2\, ,
\end{equation}
for $y=\Pi_N\, x$.  We note in particular that~\eqref{eq:droplet_x} (resp.~\eqref{eq:droplet_y}) implies that $W_2(\mu^{(N),x},\delta_{\bar{x}})=\ell$ (resp. $W_2(\mu^{(N),y},\delta_{0})=\ell$).

\subsection{Markov/sub-Markov semigroups and generators} 
\label{sub:semigroups_and_generators}
We denote by $(S_t)_{t\geq0}$, $S_t: L^\infty(\Gamma_N^\perp)\to L^\infty(\Gamma_N^\perp)$\footnote{We remark on the subtle difference between $L^\infty$ which is an equivalence class and the space of bounded, measurable functions. Since our semigroup is sufficiently regular (the law of $Y_t$ is absolutely continuous) we can define it directly on $L^\infty$.} the Markov semigroup associated to $Y_t$, 
\begin{equation}
(S_tf)(y)\,\coloneqq\,\E\left[f(Y_t)|Y_0=y\right] \, .
\label{eq:markovsg}
\end{equation}

\paragraph{1.~\textit{Properties of the Markov semigroup $S_t$.}}
It is clear from the discussion in \cref{sub:the_empirical_measure} that $\mu_N$ is invariant for the semigroup $S_t$. Since the diffusion \eqref{eq:quotientparticle} has smooth and bounded coefficents, it follows from \cite[Proposition 6.16]{Bau14} that $S_t$ is Feller and its infinitesimal generator $L$ takes the following form on any $f \in C_c^\infty(\Gamma_N^\perp)$:
\begin{equation}
Lf:= \Delta f -\nabla U_N\cdot \nabla f \, . 
\end{equation}
Furthermore, one can check that $L$ is non-positive and  is essentially self-adjoint on $C_c^\infty(\Gamma_N^\perp)$ with respect to the inner product on $L^2(\Gamma_N^\perp;\mu_N)$ (see \cite[Proposition 4.11 \& Exercise 4.12]{Bau14}) and its minimal non-positive self-adjoint extension (which we continue to denote by $L$) generates a strongly continuous contraction Markov semigroup on $L^2(\Gamma_N^\perp;\mu_N)$ (see \cite[Theorem 4.15 \& Theorem 4.25]{Bau14} and the surrounding discussion). Clearly this semigroup must agree with $S_t$ on $L^2(\Gamma_N^\perp;\mu_N)\cap L^\infty(\Gamma_N^\perp)$ and so we continue to denote it by $S_t$. By interpolation (see \cite[Theorem 4.31]{Bau14}), we can extend $S_t$ to act as a contraction Markov semigroup on $L^p(\Gamma_N^\perp;\mu_N)$ for any $1\leq  p \leq \infty$.

Given the notion of droplet introduced in \cref{sub:the_empirical_measure}, it makes sense to define the following stopping time for the process $Y_t$:
\begin{equation}
\label{eq:taudelta}
\tau_{\delta}\,\coloneqq\,\inf\left\{t\geq 0\,:\,Y_t\notin B_{\sqrt{N}\delta}\right\} \, ,
\end{equation}
where $B_{\sqrt{N}\delta}$ is the open ball of radius $\sqrt{N}\delta$ in $\Gamma_N^\perp$ for some $\delta>0$. In what follows, we will choose $\delta>0$ such that 
\begin{equation}
w(\delta) + \frac{\lambda}{2}\delta^2\,>\,0\, , 
\label{eq:deltaCondition}
\end{equation}
where $w$ is as given in \cref{ass:W}. Note that this is always possible since $w''(0)$ is assumed to be larger that $-\min_{\R}w''$, and we can choose $\lambda=\min_{\R}w''$. Indeed, by Taylor expanding we see that there exist constants $c_w,\delta'>0$ such that for all $\delta<\delta'$
\begin{equation}
w(\delta) + \frac{\lambda}{2}\delta^2\,> c_w \delta^2\, .
\label{eq:quadraticlowerbound}
\end{equation}
\\
The rationale for the condition \eqref{eq:deltaCondition}~will become apparent in the proof of Lemma \ref{lemma:depth}. 

We denote by $(P_t)_{t\geq 0}$, $P_t: L^\infty(\bar{B}_{\sqrt{N}\delta}) \to L^\infty(\bar{B}_{\sqrt{N}\delta}) $, the sub-Markov semigroup associated to $Y_t,\tau_\delta$, 
\begin{equation}
\label{eq:submarkov}
(P_tf)(y)\,\coloneqq\,\E\left[f(Y_t)\,\mathbf{1}_{t<\tau_\delta}|Y_0=y\right] \, ,
\end{equation} 
where $\bar{B}_{\sqrt{N}\delta}$ denotes the closed ball of radius $\sqrt{N}\delta$.

\vspace{1em}

\paragraph{2.~\textit{Properties of the sub-Markov semigroup $P_t$.}}
It is straightforward to check that, for $f \in C_c^\infty(\Gamma_N^\perp)$ with $\mathrm{supp\,} f$ strictly contained in $\bar{B}_{\sqrt{N}\delta}$, the generator of $P_t$ is exactly given by $L$. 
Furthermore, we denote by $p_N \in \mathcal{P}(\bar{B}_{\sqrt{N}\delta}) $  the normalized version of $\mu_N$, i.e.
\begin{equation}
\label{eq:pN}
p_N\coloneqq\left(\mu_N(\bar{B}_{\sqrt{N}\delta})\right)^{-1}\mu_N\, . 
\end{equation}

Consider the set
\begin{equation}
C_{c}^\infty(\bar{B}_{\sqrt{N}\delta}):=\left\{f \in C_c^\infty(\Gamma_N^\perp): \mathrm{supp\,} f \subseteq \bar{B}_{\sqrt{N}\delta} \right\}\, .
\end{equation}
Then, it follows from \cite[Proposition 4.50]{Bau14}, that $L$ is essentially self-adjoint on $C_{c}^\infty(\bar{B}_{\sqrt{N}\delta})$ with respect to the inner product on $L^2(\bar{B}_{\sqrt{N}\delta};p_N)$, and its minimal self-adjoint extension which we denote by $L_D$ generates a strongly continuous sub-Markov semigroup $P_t^D$ on $L^2(\bar{B}_{\sqrt{N}\delta};p_N)$.
Note that for any $f \in C_{c}^\infty(\bar{B}_{\sqrt{N}\delta})$, $P_t^Df$ is the unique solution to the Dirichlet initial value problem with initial data $f$. Applying It\^o's formula to $f(Y_t)$, one can check that this is also the case for $P_tf$. It thus follows that $P_t f= P_t^D f $ for all $f \in C_{c}^\infty(\bar{B}_{\sqrt{N}\delta})$. Additionally, from Jensen's inequality and the invariance of $\mu_N$ it follows
that
\begin{equation}
\lVert P_tf \rVert_{L^2(\bar{B}_{\sqrt{N}\delta};p_N)}\leq \lVert f \rVert_{L^2(\bar{B}_{\sqrt{N}\delta};p_N)} \, ,
\end{equation}
for all $f \in L^\infty(\bar{B}_{\sqrt{N}\delta})$. It therefore follows that $P_tf =P_t^Df$ for all $f \in L^\infty(\bar{B}_{\sqrt{N}\delta})$ and $P_t^D$ is the unique uniformly continuous extension of $P_t$ to $L^2(\bar{B}_{\sqrt{N}\delta};p_N)$. We will drop the superscript $D$ for the sake of notational convenience. Again, by interpolation (see \cite[Theorem 4.31]{Bau14}), we can extend $P_t$ to act as a  sub-Markov contraction semigroup on $L^p(\bar{B}_{\sqrt{N}\delta};p_N)$ for any $1\leq p \leq \infty$. 

\vspace{1em}

\paragraph{3.~\textit{Spectrum and heat kernel representation of $P_t$}.}
It now follows from \cite[Theorem 4.52]{Bau14} that $P_t$ is a compact operator with a discrete spectrum on $L^2(\bar{B}_{\sqrt{N}\delta};p_N)$. More precisely, there exist a complete orthonormal basis $(e_{k})_{k \in \N }$ of the space $L^2(\bar{B}_{\sqrt{N}\delta};p_N)$ and real numbers $0<\lambda_1 <\lambda_2 \leq \lambda_{3}\leq  \dots <\infty$, such that $-L_D$ has discrete spectrum with eigenfunctions and eigenvalues $e_{k}$ and $\lambda_k$, respectively. Note that $\lambda_1$ is simple \cite{P85}. 
Hence, we can define, for any $K\geq 1$, the projection operator 
\begin{equation}
\mathsf{Q}_{\geq K}: L^2(\bar{B}_{\sqrt{N}\delta};p_N) \righttoleftarrow, \qquad
\mathsf{Q}_{\geq K}f\coloneqq \sum_{k \geq K}\left(\int_{\bar{B}_{\sqrt{N}\delta}}f e_{k} \, \mathrm{d}p_N\right)e_{k} \, .
\label{eq:projection}
\end{equation}
Furthermore, $P_t$ enjoys the following heat kernel representation
\begin{equation}
(P_t f)(y)=\int_{\bar{B}_{\sqrt{N}\delta}}p_t(y,y')f(y')\, \mathrm{d}p_N(y') \, ,
\label{eq:decomposition}
\end{equation}
where $p_t \in C^\infty({B}_{\sqrt{N}\delta} \times {B}_{\sqrt{N}\delta})\cap C(\bar{B}_{\sqrt{N}\delta} \times \bar{B}_{\sqrt{N}\delta})$ is defined as follows
\begin{equation}
p_t(y,y'):=\sum_{k=1}^\infty  e^{-\lambda_k t}e_{k}(y)e_{k}(y')\, ,
\label{eq:heatkernel}
\end{equation}
with $p_t(y,y')=0$ if $y,y' \in \partial {B}_{\sqrt{N}\delta}$. Furthermore, $P_t f \in C_c^\infty (\bar{B}_{\sqrt{N}\delta})$ for all $f \in L^2(\bar{B}_{\sqrt{N}\delta};p_N)$. We use the notation $\lambda_k=\lambda_{k,\beta,N}$ when we want to emphasise the dependence on $\beta,N$.

\vspace{1em}

\subsection{One point compactification} 
\label{sub:one_point_compactification}

Consider the set $B_{\star,\sqrt{N}\delta}$, the one-point compactification of $B_{\sqrt{N}\delta}$ with some point $\star$, which we refer to as the cemetery state. 
Note that $B_{\star,\sqrt{N}\delta}$ is a compact Hausdorff space (in fact, it is homeomorphic to the sphere $\mathbb{S}^{Nd}$). Furthermore, we can equip the space with the  metric
\begin{align}
d_\star(y_1,y_2)\coloneqq& \label{eq:starmetric}
\begin{cases}
\min(|y_1-y_2|,h_\star(y_1)+h_\star(y_2)), & \, \text{ if } (y_1,y_2) \in B_{\sqrt{N}\delta}, \\
h_{\star}(y_1), & \, \text{ if }  y_2=\star,
\end{cases}\\
 h_\star(y)\coloneqq&\, \inf_{\hat y \in  \partial \bar{B}_{\sqrt{N}\delta}}|y-\hat y|\, ,
\end{align}
making it a compact metric space (and therefore Polish). In ~\cref{thm:multiscale}, we present our result in terms of the process $\hat{Y}_t$, defined as 
\begin{equation}
\hat{Y}_t\coloneqq
\begin{cases}
Y_t & \, \text{ if }  t <\tau_\delta \, , \\
\star & \, \text{ if }  t \geq \tau_\delta \, .
\end{cases}
\label{eq:killed}
\end{equation}

For a given $f :\bar{B}_{\sqrt{N}\delta} \to \R$ (resp. $f :\Gamma_N^\perp \to \R$), we will often abuse notation by using $f$ to also refer to its extension by $0$ to $\Gamma_N^\perp$ (resp. restriction to $\bar{B}_{\sqrt{N}\delta}$). It will be clear from context which object we are referring to.

\subsection{The quasi-stationary distribution} 
\label{sub:the_quasi_ergodic_and_quasi_stationary_distributions}
As we will see in \cref{thm:mathieu}~below, the bottom eigenvalue $\lambda_1$ of $-L_D$ is simple for large enough $\beta>0$. Consequently, the probability measure $q_N$, defined for measurable $A\subset \bar{B}_{\sqrt{N}\delta}$ as 
\[
q_N(A)\,\coloneqq\,\left(\int_{\bar{B}_{\sqrt{N}\delta}}e_1(y)\,\mathrm{d}p_N(y)\right)^{-1}\,\int_A e_1(y)\,\mathrm{d}p_N(y)\,,
\]
is the unique \emph{quasi-stationary distribution}~(QSD) of \eqref{eq:quotientparticle}. That is, $q_N$ is the only probability measure such that 
\[
\E_{q_N}\left[Y_t\in\,\cdot\,\rvert\,t<\tau_\delta\right]\,=\,q_N(\,\cdot\,) \qquad\text{ for }\,t\ge0\,,  
\]
where $\E_{q_N}$ denotes the expectation conditioned on the event of $Y_0$ being distributed according to $q_N$. 
In Theorem 2.6, below, we see that the probability measure $q_N$ can be used to characterise the droplet state described in Section 1 \emph{at the level of the stochastic process $Y_t$}. 
Therefore, the probability measure $\bar{\mu}^N$ referred to in Theorem 1.1 may be defined as the empirical measure of a stationary process with distribution $q_N$. 

The fact that $q_N$ is the unique QSD for $Y_t$ on  $B_{\sqrt{N}\delta}$ can be proven using the arguments of \cite[Theorem 3.1]{ZLS14}. 
The existence and uniqueness of QSDs has also been studied using spectral techniques in \cite{C09,DMV18,HK19}, as well as using modified Doeblin type conditions \cite{CV16,CV17a}, Lyapunov type conditions \cite{CV21}, and Banach lattice theory \cite{C024}. 
The ideas underlying QSDs trace back to \cite{W31,Y47}, with the first use of the term ``quasi-stationarity'' being due to \cite{B60, B57}. 

QSDs have since been widely used to characterise metastable behaviour (see the  reviews \cite{MV2012, VDP13}~for more details). Here, we use the QSD $q_N$ to characterise the metastable droplet state of \eqref{eq:quotientparticle}. In particular, we characterise the length scale of this droplet state as the variance of $\mu^{(N),Y}$, where $\mu^{(N),Y}$ is the empirical measure of some random variable $Y \sim q_N$.

\subsection{Statement of the main results} 
\label{sub:statement_of_the_main_results}
We are now prepared to rigorously and precisely state our main results on the convergence to, and leakage from, the metastable droplet state of \eqref{eq:particlesystem}. 
Our first result characterises the large $\beta$ asymptotics of the first and second eigenvalues of $-L_D$. 
We find that the first eigenvalue of $-L_D$ tends to zero at an asymptotically exponential rate as $\beta\rightarrow\infty$, with this rate being entirely determined by the behaviour of $w$ at zero and the boundary of a carefully chosen neighbourhood of zero. 

\begin{thm}[Low temperature asymptotics on the exponential scale]
Fix $N\in\N$, assume that $W$ satisfies \cref{ass:W}, and fix $\delta>0$ such that \eqref{eq:deltaCondition} is satisfied. Let $-L_D$ and $(\lambda_{i,\beta,N})_{i \geq 1}$ be as defined in \cref{sub:semigroups_and_generators}. Then, 
\begin{equation}
\lim_{\beta \to \infty}\beta^{-1}\ln \lambda_{1,\beta,N} \leq -\frac12\left(w(\delta)+\frac\lambda2 \delta^2\right) \, . 
\label{eq:d1logbound}
\end{equation}
Furthermore, for all $i \geq 2$, we have
\begin{equation}
\lim_{\beta \to \infty}\beta^{-1}\ln \lambda_{i,\beta,N}=0 \, .
\end{equation}
\label{thm:mathieu}
\end{thm}

\begin{rmk}
We remark that \cref{thm:mathieu} is true even if \eqref{eq:deltaCondition} does not hold, as will become clear from the proof in \cref{sec:proof_of_thm:mathieu}. However, in this case the bound \eqref{eq:d1logbound} does not contain any useful information.
\end{rmk}
\begin{ex}
As an example of a one-dimensional potential $W$ which satisfies~\eqref{eq:deltaCondition}, consider $W(x)=w(x)=1-e^{-x^2}$. Note that this function is $\lambda$-convex with $\lambda$ chosen to be $-4e^{-\frac32}$. We thus have
\begin{equation}
\frac{\mathrm{d}}{\mathrm{d}x}\left(w(x) + \frac{\lambda}{2}x^2\right)=2xe^{-x^2}-4e^{-\frac32}x \, .
\end{equation}
The above expression is strictly greater than $0$ for $ 0 < x <\delta'\coloneq \sqrt{3/2 -\ln 2}$. Thus,~\eqref{eq:deltaCondition} is true for all $\delta <\delta'$. 
\end{ex}

Our next main result refines the asymptotic description of $\lambda_{2,\beta,N}$ in \eqref{eq:}. 
This result implies a bound away from zero of each $\lim_{\beta \to \infty}\lambda_{2,\beta,N}$ 
which is uniform in $N\in\N$. 
As seen in Section \ref{sec:proof_of_thm:simon}, we use the fact that $-L_D$ is unitarily equivalent to a certain Schrödinger operator and study the spectrum of this operator borrowing ideas from~\cite{Simon1983}. 

\begin{thm}[Refined low temperature asymptotics]
Under the assumptions of Theorem \ref{thm:mathieu}, the following convergence holds:
\begin{align}
\lim_{\beta\rightarrow\infty}\lambda_{1,\beta,N}=&\,0\,, \\
\lim_{\beta\rightarrow\infty}\lambda_{2,\beta,N}=&\,w''(0)\,. 
\end{align}
\label{thm:simon}
\end{thm}

Now that we have identified the large $\beta$ asymptotics of the spectrum of $-L_D$, we are able to characterise the dynamical metastability of \eqref{eq:particlesystem}~under the assumptions used in this paper. 
We see in particular that $\lambda_{2,\beta,N}$ describes the rate of convergence of \eqref{eq:particlesystem}~to a droplet state, whereas $\lambda_{1,\beta,N}$ describes the rate at which particles escape the droplet. 
Additionally, the droplet state itself is characterised by the QSD $q_N$ (or rather, the random empirical measure sampled from $q_N$), and we demonstrate that this droplet state has size (on average) $\beta^{-\frac12}$.  \cref{thm:multiscale} below provides a precise description of this phenomenon.~\cref{cor:TimeScales}, which translates~\cref{thm:multiscale} to the language of time scales, is then almost immediate.  Specifically, we find that for each large but fixed $\beta$, there is an interval of times determined by $\lambda_{1,\beta,N}$ and $\lambda_{2,\beta,N}$ over which the distribution of $Y_t$ is close to the QSD $q_N$ (in either total variation or $1$-Wasserstein distance). 

\begin{thm} \label{thm:multiscale}
Consider the process $\hat{Y_t}$ as defined in~\eqref{eq:killed} for some $\delta>0$.
\begin{enumerate}[\normalfont(i)]
    \item For any $f \in C(B_{\star,\sqrt{N}\delta})$ (see~\cref{sub:one_point_compactification}) and $\nu_N \in \cP(\bar{B}_{\sqrt{N}\delta})$, the following identity holds true:
\begin{align}
&\E\left[f(\hat{Y}_t)|Y_0\sim \nu_N\right]-\alpha_{\beta,N} e^{-\lambda_{1,\beta,N}t}\left(\int_{\bar{B}_{\sqrt{N}\delta}}f \,\mathrm{d}q_N \right) - (1-\alpha_{\beta,N} e^{-\lambda_{1,\beta,N}t}) f(\star) \\
&\qquad=\int_{\bar{B}_{\sqrt{N}\delta}}(P_t \mathsf{Q}_{\geq 2}(f-f(\star)))\, \mathrm{d}\nu_N\, ,\label{eq:multiscaleidentity}
\end{align}
where $\mathsf{Q}_{\geq 2}$ is the projection defined in~\eqref{eq:projection} and 
\begin{equation}
\alpha_{\beta,N}\coloneqq\left(\int_{\bar{B}_{\sqrt{N}\delta}}e_1\, \mathrm{d}\nu_N\right)\left(\int_{\bar{B}_{\sqrt{N}\delta}}e_1 \, \mathrm{d}p_N \right)\,.
\label{eq:prefactor}
\end{equation}

\item Let $\rho_t^{\nu_N}:=\mathrm{Law}(\hat{Y}_t|Y_0\sim \nu_N)$ and recall $p_N$~\eqref{eq:pN}.
Then, given~\eqref{eq:multiscaleidentity}, we can derive the following estimates.
\begin{itemize}
\item Control in $\mathrm{TV}$: Assume that $\nu_N \ll p_N$, then we have the bound
\begin{equation}
\lVert \rho_t^{\nu_N} -  \alpha_{\beta,N} e^{-\lambda_{1,\beta,N}t}q_N - (1-\alpha_{\beta,N} e^{-\lambda_{1,\beta,N}t}) \delta_\star \rVert_{\rm TV} \leq 2 e^{-\lambda_{2,\beta,N}t} \left \lVert \frac{\mathrm{d} \nu_N}{\mathrm{d}p_N}\right\rVert_{L^2(\bar{B}_{\sqrt{N}\delta};p_N)}\, .
\label{eq:TVbound}
\end{equation}
\item Control in $W_1$:  Assume again that $\nu_N \ll p_N$, then we have the bound
\begin{equation}
W_1\left(\rho_t^{\nu_N} ,  \alpha_{\beta,N} e^{-\lambda_{1,\beta,N}t}q_N + (1-\alpha_{\beta,N} e^{-\lambda_{1,\beta,N}t}) \delta_\star \right) \leq \sqrt{N}\delta e^{-\lambda_{2,\beta,N}t} \left \lVert \frac{\mathrm{d} \nu_N}{\mathrm{d}p_N}\right\rVert_{L^2(\bar{B}_{\sqrt{N}\delta};p_N)}\, ,
\label{eq:W1bound}
\end{equation}
where $W_1$ is defined with respect to the metric $d_\star$ defined in~\eqref{eq:starmetric}.
\end{itemize}
\item We have that
\begin{equation}
\lim_{\beta \to \infty} \alpha_{\beta,N}=1\, .
\label{eq:alphaconvergence}
\end{equation}
\item If $\mu^{(N),y}$ denotes an empirical measure associated to the configuration $y \in \Gamma_N^\perp$ and $Y \sim q_N$, then
\begin{equation}
\limsup_{\beta \to \infty}\beta\mathbb{E}\left[\int_{\R^d}|x|^2 \, \mathrm{d}\mu^{(N),Y}\right] \lesssim 1\, .
\label{eq:dropletsize}
\end{equation}
\end{enumerate}
\end{thm}

From Theorem \ref{thm:multiscale}, we then have the following. 

\begin{cor}
\label{cor:TimeScales}
Taking $\nu_N\in\mathcal{P}(\bar{B}_{\sqrt{N}\delta})$, such that $\nu_N\ll p_N$, for any increasing family of times $(t_\beta)_{\beta>0}$ satisfying 
\begin{equation}
\label{eq:tscale}
\lambda_{1,\beta,N}(\beta) t_\beta\,\xrightarrow[\beta\rightarrow\infty]{}\,0 \qquad\text{ and }\qquad 
\lambda_{2,\beta,N}(\beta)t_\beta\,\xrightarrow[\beta\rightarrow\infty]{}\,\infty,
\end{equation}
it holds that 
\begin{equation}
\label{eq:qsdscale}
\lim_{\beta \to \infty}\norm*{\rho_{t_\beta}^{\nu_N}-q_N}_{\rm TV}=0 \, .
\end{equation}
\begin{proof}
Using Theorem 2.6, observe that 
\[
\begin{aligned}
\norm*{\rho_{t_\beta}^{\nu_N}-q_N}_{\rm TV}\,&\le\,\norm*{\rho_{t_\beta}^{\nu_N} - \left(\alpha_\beta e^{-\lambda_{1,\beta,N} t_\beta}q_N + (1-\alpha_\beta e^{-\lambda_{1,\beta,N} t_\beta}) \delta_\star \right)}_{\rm TV} \\
&\qquad+ \norm*{q_N- \left( \alpha_\beta e^{-\lambda_{1,\beta,N} t_\beta}q_N + (1-\alpha_\beta e^{-\lambda_{1,\beta,N} t_\beta}) \delta_\star \right)}_{\rm TV}\\ 
&\le\, 2 e^{-\lambda_{2,\beta,N}t_\beta} \left \lVert \frac{\mathrm{d} \nu_N}{\mathrm{d}p_N}\right\rVert_{L^2(\bar{B}_{\sqrt{N}\delta};p_N)} + \left(1-\alpha_\beta e^{-\lambda_{1,\beta,N}t_\beta}\right)\norm*{q_N-\delta_\star}_{\rm TV}. 
\end{aligned}
\]
Since $\alpha_\beta\rightarrow1$ as $\beta\rightarrow\infty$, the above quantity converges to zero as $\beta\rightarrow\infty$. 
An analogous computation holds for the $1$-Wasserstein distance. 
\end{proof}
\end{cor}

\section{Proof of \texorpdfstring{\cref{thm:mathieu}}{exponential asymptotics}} 
\label{sec:proof_of_thm:mathieu}
The proof of \cref{thm:mathieu} relies on a result due to Mathieu (see \cite[Theorem 1]{M95}) which studies the asymptotics of operators which have the same form as $L_D$ in the regime of small noise. Before we can recall this result, we need to introduce some additional notions.

Given a smooth potential $U_N: \bar{B}_{\sqrt{N}\delta}\to \R$, we say that a subset $A \subseteq \Gamma_N^\perp$ with smooth boundary is an $r$-valley, for some $r>0$, if it is a connected component of the set $\{y \in \Gamma_N^\perp: U_N(y) < a\}\cap \bar{B}_{\sqrt{N}\delta}$ for some $a \in \mathbb{R}$, and furthermore
\begin{equation}
\inf_{\partial A \setminus \partial {B}_{\sqrt{N}\delta}} U_N -\inf_A U_N=r \, .
\end{equation}
We denote by $f(r)$ the number of $r$-valleys associated to $U_N$ and define for $i \geq 1$
\begin{equation}
d_i= \inf\{r>0: f(r)<i\} \, .
\end{equation}
The result of Mathieu then takes the following form in our setting.
\begin{thm}[{\cite[Theorem 1]{M95}}]
Let $-L_D$ be as constructed in \cref{sub:semigroups_and_generators} and $\lambda_{i,\beta,N}, i\geq 1$ be its eigenvalues. Then,
\begin{equation}
\lim_{\beta \to \infty} \beta^{-1}\ln\lambda_{i,\beta,N} = -\frac{d_i}{2} \, .
\end{equation}
\label{thm:mathieucopy}
\end{thm}

We proceed to show that the Hamiltonian $U_N$ has a unique minimum and critical point at $0$.

\begin{lemma}
Let $W$ satisfy \cref{ass:W}. Then, the function $U_N$ possesses a unique minimum and critical point at $0$ with $U_N(0)=0$. Furthermore, for any $a \in \R$, the set
\begin{equation}
\bar{A}=\{y \in \Gamma_N^\perp: U_N(y) < a\}\cap \bar{B}_{\sqrt{N}\delta}
\end{equation}
is path-connected and, thereby, connected.
\label{lemma:mincrti}
\end{lemma}
\begin{proof}
We start by noting that $H_N(x)=0$ if $x \in \Gamma_N$ , the orthogonal complement of $\Gamma_N^\perp$. Furthermore, if $x \notin \Gamma_N$, we can write $x=z+y$ for some $z \in \Gamma_N$ and some $0\neq y \in \Gamma_N^\perp$. This implies that there exists some $(i_0,j_0)\in \{1,\dots,N\}^2$ such that $y_{i_0}\neq y_{j_0}$. Using \cref{ass:W}, we then have 
\begin{align}
H_N(x)=& \,H_N(y)\\
\geq& \, \frac{1}{N}w(|y_{i_0}-y_{j_0}|)>0 \, .
\end{align}
We recall that $U_N=\left.H_N\right|_{\Gamma_N^\perp}$. Hence, we have $U_N(0)=0$ and, for any $0\neq y\in \Gamma_N^\perp$, 
\begin{equation}
U_N(y)=H_N(y)>0\,.
\end{equation}
This establishes the uniqueness of the minimum at $0$. 

To argue that $0$ is also the unique critical point, 
we observe from the properties of $W$ in \cref{ass:W} that
\begin{align}
\langle \nabla H_N(x),x\rangle = & \,\frac{1}{N}\sum_{i,j=1}^N w'(|x_i-x_j|)\frac{x_i-x_j}{|x_i-x_j|}\cdot x_i  \, ,
\end{align}
from which it follows by flipping the roles of $i$ and $j$ and summing that 
\begin{align}
\langle \nabla H_N(x),x\rangle = & \,\frac{1}{2N}\sum_{i,j=1}^N w'(|x_i-x_j|)|x_i-x_j|>0  \, ,
\end{align}
if $x \notin \Gamma_N$. It follows from the definition of $U_N$ that
\begin{equation}
\langle \nabla U_N(y), y \rangle>0 \label{eq:1criticalpoint} \, ,
\end{equation}
for $y \neq 0$. This establishes the uniqueness of $0$ as a critical point. 

For path-connectedness, we argue that, for any $(y_0,y_1)\in \bar{A}$, the curve
\begin{align}
y(t)=
\begin{cases}
(1- 2t )y_0, & \text{ if } t \in [0,\frac12]\\
(2t-1) y_1, & \text{ if } t \in (\frac12,1]
\end{cases}
\, ,
\end{align}
lies in $\bar{A}$. This is clearly true for $\bar{B}_{\sqrt{N}\delta}$. For  $y_0\in\{y \in \Gamma_N^\perp:U_N(y)< a\}$, we note that
\begin{equation}
U_N(y_0)=U_N(\alpha y_0) + \int_\alpha^1 \langle \nabla U_N(\alpha'y_0), y_0\rangle \, \mathrm{d}\alpha'\geq U_N(\alpha y_0) \, ,
\end{equation}
as long has $\alpha\leq 1$. From this it follows that $\alpha y_0 \in\{y \in \Gamma_N^\perp:U_N(y)< a\}$ for any $\alpha \in [0,1]$. The same holds for $y_1$ such that path-connectedness follows.
\end{proof}
As an immediate corollary of the above lemma, we have the following result.
\begin{cor}\label{cor:dig2}
Let $W$ satisfy \cref{ass:W}. Then, for all $i\geq 2$, $d_i=0$.
\end{cor}
\begin{proof}
The proof is a straightforward consequence of the fact that the set $\bar{A}$ in the statement of \cref{lemma:mincrti} is  connected, since it implies that $f(r)<2$ for all $r > 0$, and so $d_i=0$ for all $i\geq 2$.
\end{proof}
It remains to establish a lower bound on $d_1$. To do this, we note, from the discussion in \cite[pg. 4]{M95},  the following alternative characterisation of $d_1$:
\begin{equation}
d_1 = \sup_{y_0 \in \bar{B}_{\sqrt{N}\delta}} \inf_{y_1 \in \partial\bar{B}_{\sqrt{N}\delta}}\inf_{\substack{\phi \in C([0,1];\bar{B}_{\sqrt{N}\delta})\\\phi(0)=y_0,\phi(1)=y_1}} \sup_{t\in [0,1]}U_N(\phi(t))-U_N(\phi(0))\, .
\end{equation}
Using the fact that $U_N(0)=0$, we clearly have the lower bound
\begin{equation}
d_1 \geq \inf_{y_1 \in \partial\bar{B}_{\sqrt{N}\delta}} U_N(y_1) \, .
\label{eq:d1lowerbound}
\end{equation}
We obtain the following result.
\begin{lemma}
\label{lemma:depth}
 Let $W$ satisfy \cref{ass:W} and fix $\delta>0$.  Then  
\begin{equation}
\begin{aligned}
\inf_{y\in \partial B_{\sqrt{N}\delta}}U_N(y)\,\geq\,\,w(\delta) + \frac{\lambda}{2}\delta^2 \, , 
\end{aligned}
\label{eq:lowerbound}
\end{equation}
where $\lambda$ is as in~\eqref{eq:lambdaconvexity}. As a consequence, we have the  lower bound
\begin{equation}
d_1 \geq w(\delta) + \frac{\lambda}{2}\delta^2 \, .
\end{equation}
\end{lemma}
\begin{proof}
 For the proof of \eqref{eq:lowerbound}, we note that $x \in \partial \bar{B}_{\sqrt{N}\delta}$ implies the existence of an $i_0\in\{1, \dots,N\}$ such that $\abs*{x_{i_0}}\ge\delta$. Thus, using \cref{ass:W} and the expression~\eqref{eq:lconvexN}, we have 

\begin{align}
U_N(x)= & \,H_N(x)\\\ge & \,\frac{1}{N}\sum_{j=1}^NW(x_{i_0}-x_j) \\
\geq &\,\,W\left(\frac{1}{N}\sum_{j=1}^Nx_{i_0}-\frac{1}{N}\sum_{j=1}^Nx_j\right)+\frac{\lambda(N-1)}{2N^2}\sum_{j=1}^N|x_{i_0} -x_j|^2 \\
\geq & \, w(\delta) +\frac{\lambda(N-1)}{2N^2}(N+1)\delta^2 \\
\geq & \, w(\delta) +\frac{\lambda}{2}\delta^2 \, ,
\end{align}
where we have used the fact that $\lambda<0$. We can combine the above bound with \eqref{eq:d1lowerbound} to conclude the proof of the lemma.
\end{proof}
Combining the results of \cref{thm:mathieucopy,cor:dig2,lemma:depth}, we have a complete proof of \cref{thm:mathieu}.

\section{Proof of \texorpdfstring{\cref{thm:simon}}{refined asymptotics}} 
\label{sec:proof_of_thm:simon}
In the following, for a Banach space $B$, if $x\in B$ then $\norm*{x}_{B}$ denotes the the norm of $x$, while if $A:B \to B$ is an operator, then $\norm*{A}_{B}=\norm*{A}_{B \to B}$ denotes the operator norm of $A$.  To study the asymptotic behaviour of the higher eigenvalues of $-L_D$, we adopt the approach of \cite{Simon1983}. To this end, we introduce a Schr{\"o}dinger operator $S_\beta$, which is unitarily equivalent to the sub-Markov generator $-L_D$, 
given by 
\[
S_\beta=-\beta^{-1}\Delta + \frac{\beta}{4}|\nabla U_N|^2 - \frac{1}{2}\Delta U_N\,=:-\beta^{-1}\Delta + V_\beta, 
\]
with Dirichlet boundary conditions on $B_{\sqrt{N}\delta}$. 
A Schr{\"o}dinger operator of this form is sometimes referred to as a \emph{Witten Laplacian}~\cite{HN05,HN06,W82}. 
Recall that $-L_D$ is defined on $L^2(B_{\sqrt{N}\delta};p_N)$, where $p_N$ is as defined in \eqref{eq:muN} and \eqref{eq:pN}, and let the density of $p_N$ be denoted by $u_N$.  One can check that  
\[
T:L^2(B_{\sqrt{N}\delta}; p_N)\,\rightarrow L^2(B_{\sqrt{N}\delta}),\qquad Tf\,\coloneqq\, u^{1/2}f \,,
\]
is a unitary conjugacy between $-L_D$ as an operator on $L^2(B_{\sqrt{N}\delta};p_N)$ and $S_\beta$ as an operator on $L^2(B_{\sqrt{N}\delta})$: 
\begin{equation}
\label{eq:conjugacy}
-L_D f\,=\,T^{-1}\left( S_\beta (Tf) \right) = u_N^{-1/2} S_\beta \left( u_N^{1/2}f  \right), \qquad f \in L^2(B_{\sqrt{N}\delta};p_N)\,. 
\end{equation}
It thus follows that $S_\beta$ has domain $u_N^{\frac12}D(-L_D)\subseteq L^2(B_{\sqrt{N}\delta})$, is self-adjoint, and $S_\beta,-L_D$ have the same spectrum. 
The analysis of Schr\"odinger operators in~\cite{Simon1983} is presented on the whole space. Thus, to apply the analysis in \cite{Simon1983}, we  need to establish a connection between the spectra of operators on $\Gamma_N^\perp$ and operators on $B_{\sqrt{N}\delta}$ with Dirichlet boundary conditions.    Moreover, we recall that we may identify $\Gamma_N^\perp$ with $\R^{(N-1)d}$, and so consider all Schr{\"o}dinger operators to be defined on (a subset of) Euclidean space.

\subsection{Approximating Schr{\"o}dinger operators on the whole space}
As mentioned earlier, to study the spectrum of $S_\beta$, we apply the results of \cite{Simon1983}, which are only stated for Schr{\"o}dinger operators defined on $\R^k$, $k\in\N$.  As \cite{Simon1983}~employs scaling arguments to study the asymptotics of the spectra of Schr{\"o}dinger operators, the results therein do not trivially transfer to Schr{\"o}dinger operators defined on bounded domains. 
To impose boundary conditions in a way that is compatible with \cite{Simon1983}, we approximate $S_\beta$ by a sequence of operators over $\R^{(N-1)d}$; the spectra of these operators will be shown to converge to the spectrum of $S_\beta$.

The convergence of sequences of differential operators on varying spatial domains has been widely studied, for instance in \cite{AD08,BD06,D08}.  However, issues often arise in the case where the spatial domains tend to an unbounded set.  To circumvent these issues, we use ideas found in \cite{DPL10}, and excellently explained in the thesis \cite{LBG}, where Schr{\"o}dinger operators on bounded domains are approximated by Schr{\"o}dinger operators on the whole space in a strong resolvent sense. 
As we want to approximate the spectrum of $S_\beta$ by the spectra of operators $\tilde{S}_\beta$, and it is known that strong resolvent convergence alone may allow for spectral pollution as $n\rightarrow\infty$ \cite[Chapter IV.3.1]{K13}, this result as such is insufficient. 
Hence, in the following, we strengthen the results of \cite{DPL10}~to norm resolvent convergence. 

In this section, for $f:B_{\sqrt{N}\delta}\rightarrow\R$ we denote by $\overline{f}$ the extension of $f$ to $\R^{(N-1)d}$ by setting $\overline{f}=0$ on $\bar{B}_{\sqrt{N}\delta}^c $.  Similarly, for $f:\R^{(N-1)d}\rightarrow\R$, we denote by $f\rvert$ the restriction of $f$ to $B_{\sqrt{N}\delta}$. Let $(B_t)_{t\ge0}$ denote a standard Brownian motion on $\R^{(N-1)d}$. We denote by $\bP_x$ the Wiener measure corresponding to $(B_t)_{t\ge0}$ with $B_0=x$, and by $\E_y$ the corresponding expectation. By the Feynman--Kac formula, the semigroup generated by $S_\beta$ with Dirichlet boundaries on $B_{\sqrt{N}\delta}$ can be expressed as 
\[
(e^{-S_\beta t}f)(y)\,\coloneqq\,\E_y\left[f(B_{\beta^{-1}t})e^{-\beta\int_0^{\beta^{-1}t}V_\beta(B_s)\,\mathrm{d}s}1_{\{\beta^{-1}t<\tau\}}\right]\,,
\]
where $\tau$ is the first exit time of $(B_t)_{t\ge0}$ from $B_{\sqrt{N}\delta}$ \cite[Lemma 2.2]{L98}~(note that the $\tau$ introduced here is distinct from $\tau_\delta$, defined in \eqref{eq:taudelta}).  
Let $V_0:\R^{(N-1)d}\rightarrow[0,1]$ be a smooth function such that 
\begin{enumerate}[1.]
\item $V_0$ is non-decreasing in $|y|$, 
\item $V_0(y)=0$ for $y\in B_{\sqrt{N}\delta}$, $V_0(y)>0$ for $y\in \bar{B}_{\sqrt{N}\delta}^c$, and 
\item $\lim_{|y|\to \infty}V_0(y)=+\infty$. 
\end{enumerate}
We then define $\tilde{S}_\beta$ on $C^\infty_c(\R^{(N-1)d})$ as 
\begin{equation}
\tilde{S}_\beta\,\coloneqq\,-\beta^{-1}\Delta + V_\beta + \beta V_0\,. 
\end{equation} 
One can readily check that $\tilde{S}_\beta$ is essentially self-adjoint on $L^2(\R^{(N-1)d})$ (see~\cite[Theorem X.29]{RS75}) and we continue to denote by $\tilde{S}_\beta$ its unique self-adjoint extension. One can also check that $\tilde{S}_\beta$ is a non-negative operator. As before, $\tilde{S}_\beta$ is the generator in the $L^2(\R^{(N-1)d})$-topology of the semigroup 
\begin{equation}
\label{eq:Pn}
e^{-\tilde{S}_\beta t}f(y)\,\coloneqq\,\E_y\left[f(B_{\beta^{-1}t})e^{-\beta\int_0^{{\beta^{-1}t}}V_\beta(B_s)\,\mathrm{d}s}e^{-\beta^2 \int_0^{\beta^{-1}t}V_0(B_s)\,\mathrm{d}s}\right]\,. 
\end{equation}

As mentioned above, the goal here is to prove that $\tilde{S}_\beta$ converges to $S_\beta$ in norm resolvent as $\beta\to \infty$.  To achieve this, we first prove the norm convergence of the semigroups $e^{-\tilde{S}_\beta t}$ to $e^{-S_\beta t}$ as $\beta \to \infty$ for fixed $t>0$ in \cref{lemma:SemigroupConvergence}.  Then, in \cref{lemma:Resolvent1}, we use the Laplace transform representation of the resolvent to prove norm convergence of the resolvents $R(\lambda,-\tilde{S}_\beta)$ to $R(\lambda,-S_\beta)$ for all $\lambda>0$ where the resolvent operators are given by 
\[
R(\lambda,-S_\beta)\,=\,(\lambda+S_\beta)^{-1}\,,\qquad \lambda\in\C\,\backslash\spec(-S_\beta)\,, 
\]
and similarly for $S_\beta$ replaced with $\tilde{S}_\beta$.  
Finally, we extend this norm resolvent convergence to every admissible $\lambda\in\C$ in \cref{prop:ResolventConvergence}~using the Stone--Weierstrass theorem. 
Note that in what follows, we will abuse notation by using $L^2(\R^{(N-1)d})$ to denote the space of complex-valued square-integrable functions over $\R^{(N-1)d}$.

\begin{lemma}
\label{lemma:SemigroupConvergence}
For $f\in L^2(\R^{(N-1)d})$ and $t>0$, it holds that 
\begin{equation}
\label{eq:}
\lim_{\beta \to \infty}\sup_{\norm*{f}_{L^2(\R^{(N-1)d})} \leq 1}\norm*{\overline{e^{-S_\beta t}(f\rvert)}-e^{-\tilde{S}_\beta t}f}_{L^2(\R^{(N-1)d})}\,=\,0 \, .
\end{equation}
\end{lemma}

\begin{proof}
For an arbitrary $f\in L^2(\R^{(N-1)d})$ and $\gamma \in (0,1)$, we have
\begin{align}
&\norm*{\overline{e^{-S_\beta t}(f\rvert)}-e^{-\tilde{S}_\beta t}f}_{L^2(\R^{(N-1)d})}^2\\
=& \,\int_{\mathbb{R}^{(N-1)d}}\left(\E_y\left[f(B_{\beta^{-1}t})e^{-\beta\int_0^{\beta^{-1}t}V_\beta(B_s)\,\mathrm{d}s}e^{-\beta^2\int_0^{\beta^{-1}t}V_0(B_s)\,\mathrm{d}s} 1_{\{\beta^{-1}t\ge\tau\}}\right]\right)^2\, \mathrm{d}y\\
\le & \,\int_{\R^{(N-1)d}}\E_y\left[|f(B_{\beta^{-1}t})|^2\right]\E_y\left[e^{-2\beta\int_0^{\beta^{-1}t}V_\beta(B_s)\,\mathrm{d}s}e^{-2\beta^2\int_0^{\beta^{-1}t}V_0(B_s)\,\mathrm{d}s} 1_{\{\beta^{-1}t\ge\tau\}}\right]\, \mathrm{d}y  \\
\le & \,\int_{B_{\gamma \sqrt{N}\delta}}\E_y\left[|f(B_{\beta^{-1}t})|^2\right]\E_y\left[e^{-2\beta\int_0^{\beta^{-1}t}V_\beta(B_s)\,\mathrm{d}s}e^{-2\beta^2\int_0^{\beta^{-1}t}V_0(B_s)\,\mathrm{d}s} 1_{\{\beta^{-1}t\ge\tau\}}\right]\, \mathrm{d}y\\
&\, +\int_{B_{\gamma^{-1} \sqrt{N}\delta}\setminus B_{\gamma \sqrt{N}\delta}}\E_y\left[|f(B_{\beta^{-1}t})|^2\right]\E_y\left[e^{-2\beta\int_0^{\beta^{-1}t}V_\beta(B_s)\,\mathrm{d}s}e^{-2\beta^2\int_0^{\beta^{-1}t}V_0(B_s)\,\mathrm{d}s} 1_{\{\beta^{-1}t\ge\tau\}}\right]\, \mathrm{d}y \\
&\, +\int_{B_{\gamma^{-1} \sqrt{N}\delta}^c}\E_y\left[|f(B_{\beta^{-1}t})|^2\right]\E_y\left[e^{-2\beta\int_0^{\beta^{-1}t}V_\beta(B_s)\,\mathrm{d}s}e^{-2\beta^2\int_0^{\beta^{-1}t}V_0(B_s)\,\mathrm{d}s} 1_{\{\beta^{-1}t\ge\tau\}}\right]\, \mathrm{d}y\\
\leq & \, e^{C_{\Delta}t}\left(\int_{\R^{(N-1)d}}\E_y\left[|f(B_{\beta^{-1}t})|^2\right]\, \mathrm{d}x\right)\bigg(\esup\limits_{y\in B_{\gamma\sqrt{N}\delta}}\mathbb{P}_y\left[\tau\leq \beta^{-1}t\right]  \\&\, +\esup_{y \in B_{\gamma^{-1} \sqrt{N}\delta}\setminus B_{\gamma \sqrt{N}\delta}}\E_y \left[e^{-\frac{\beta^2}{4} \int_0^{\beta^{-1}t}|\nabla U_N|^2(B_s) \, \mathrm{d}s} \right] +\esup_{y\in B_{\gamma^{-1} \sqrt{N}\delta}^c}\E_y \left[e^{-\beta^2 \int_0^{\beta^{-1}t}V_0(B_s) \, \mathrm{d}s} \right] \bigg)\\
\leq & \, e^{C_\Delta t}\lVert f\rVert_{L^2(\R^{(N-1)d})}^2\bigg(\esup\limits_{x\in B_{\gamma\sqrt{N}\delta}}\mathbb{P}_y\left[\tau\leq \beta^{-1}t\right]  \\&\, +\esup_{y \in B_{\gamma^{-1} \sqrt{N}\delta}\setminus B_{\gamma \sqrt{N}\delta}}\E_y \left[e^{-\frac{\beta^2}{4} \int_0^{\beta^{-1}t}|\nabla U_N|^2(B_s) \, \mathrm{d}s} \right] +\esup_{y \in B_{\gamma^{-1} \sqrt{N}\delta}^c}\E_y \left[e^{-\beta^2 \int_0^{\beta^{-1}t}V_0(B_s) \, \mathrm{d}s} \right] \bigg)\, , \label{eq:termstobeest}
\end{align}

where we have used the fact that the heat semigroup is non-expansive in $L^2(\R^{(N-1)d})$, $|\Delta U_N|\leq C_\Delta<\infty$, and $V_0\geq 0$. We now need to control the three terms that show up in the bracket on the right hand side of~\eqref{eq:termstobeest}. For the first term, we note that 
\begin{equation}
\esup_{x\in B_{\gamma \sqrt{N}\delta}}\mathbb{P}_y\left[\tau\leq \beta^{-1}t\right]\, \leq\, \mathbb{P}_{\bar{y}}\left[\tau\leq \beta^{-1}t\right]
\label{eq:stoppingbound}
\end{equation}
for some arbitrary $\bar{y}\in \partial B_{\gamma\sqrt{N}\delta}$ (due to the rotational invariance of Brownian motion this probability is independent of the choice of $\bar{y}$). 
The term on the right hand side goes to $0$ as $\beta \to \infty$ since $\tau>0$ almost surely for $\bar{y}\in \partial B_{\gamma\sqrt{N}\delta}$. 
To see that~\eqref{eq:stoppingbound} is true, let $\tau'$ denote the first exit time from $B_{\gamma\sqrt{N}\delta}$. 
Then, we argue that, for $y\in B_{\gamma \sqrt{N}\delta}$, 
\begin{align}
\mathbb{P}_{y}[\tau\leq t] =&\,\mathbb{P}_{y}\left[\sup_{0\leqslant s\leqslant t}|B_{s}|\geq\sqrt{N}\delta\right]\\
   =&\,\mathbb{P}_{y}\left[\sup_{\tau'\leq s\leq t}|B_{s}|\geq \sqrt{N}\delta\right]\\
   =&\,\mathbb{P}_{y}\left[\sup_{0\leq s\leq t-\tau'}|B_{s+\tau'}|\geq \sqrt{N}\delta\right]\\
   =&\,\mathbb{P}_{y}\left[\sup_{0\leq s\leq t-\tau'}|B_{s+\tau'}-B_{\tau'}+B_{\tau'}|\geq\sqrt{N}\delta\right]\, . 
\end{align}
Note now that by the strong Markov property $B_{s+\tau'}-B_{\tau'}$ is a Brownian motion itself and is independent of $\mathcal{F}_{\tau'}$. 
As a consequence, we can rewrite the above expression as
\begin{align}
&\int_{\partial B_{\gamma \sqrt{N}\delta}\times[0,\infty)}\mathbb{P}_{0}\left[\sup_{0\leq s\leq t-r}|B_{s}+y|\geq \sqrt{N}\delta \right]\mathrm{d}\mu(y,r)\\
=&\,  \int_{[0,\infty)}\mathbb{P}_{0}\left[\sup_{0\leqslant s\leqslant t-r}|B_{s}+\bar{y}|\geq \sqrt{N}\delta\right]\mathrm{d}\tilde{{\mu}}(r)
\leq  \mathbb{P}_{\bar{y}}[\tau\leq t]\,,
\end{align}
where $\bar{y}\in\partial B_{\gamma \sqrt{N}\delta}$ is arbitrary, $\mu$  is the joint law of $\left(B_{\tau'},\tau'\right)$, and $\tilde{{\mu}}$ is the law of $\tau'$. The last inequality follows from the fact that $\{\tau\leq t-r\}\subseteq\{\tau\leq t\}$ for all $r\ge0$. 

We now control the second term in the bracket on the right hand side of~\eqref{eq:termstobeest}. By~\cref{lemma:mincrti}, we know that 
\begin{equation}
\inf_{y\in B_{2\gamma^{-1} \sqrt{N}\delta}\setminus B_{(\gamma/2) \sqrt{N}\delta}}|\nabla U_N|^2(y)\,\eqqcolon\,C_\nabla\,>\,0\, .
\end{equation}
Let $\tilde{\tau}_1$ denote the first exit time from $B_{2(\gamma)^{-1} \sqrt{N}\delta}$, $\tilde{\tau}_2$ the first hitting time for $B_{(\gamma/2)\sqrt{N}\delta}$, and $\tilde{\tau}\coloneqq\tilde{\tau}_1\wedge\tilde{\tau}_2$. Then, we have that
\begin{align}
&\, \esup_{y \in B_{\gamma^{-1} \sqrt{N}\delta}\setminus B_{\gamma \sqrt{N}\delta}}\E_y \left[e^{-\frac{\beta^2}{4} \int_0^{\beta^{-1}t}|\nabla U_N|^2(B_s) \, \mathrm{d}s} \right]\\
\leq & \, \esup_{y \in B_{\gamma^{-1} \sqrt{N}\delta}\setminus B_{\gamma \sqrt{N}\delta}} \E_y \left[e^{-\frac{\beta C_\nabla t }{8} } \right] \\&\,+ \esup_{y \in B_{\gamma^{-1} \sqrt{N}\delta}\setminus B_{\gamma \sqrt{N}\delta}}\E_y \left[\mathbf{1}_{\tilde{\tau}\leq \beta^{-1}t/2} \right]\\
\leq &\, e^{-\frac{\beta C_\nabla t }{8} } +\esup_{y \in B_{\gamma^{-1} \sqrt{N}\delta}\setminus B_{\gamma \sqrt{N}\delta}}\mathbb{P}_y \left[\tilde{\tau}\leq \frac{\beta^{-1}t}{2} \right]\\
\leq & \,e^{-\frac{\beta C_\nabla t }{8} }+\esup_{y \in B_{\gamma^{-1} \sqrt{N}\delta}\setminus B_{\gamma \sqrt{N}\delta}}\mathbb{P}_y \left[\tilde{\tau}_1\leq \frac{\beta^{-1}t}{2} \right]\\&\,+\esup_{y \in B_{\gamma^{-1} \sqrt{N}\delta}\setminus B_{\gamma \sqrt{N}\delta}}\mathbb{P}_y \left[\tilde{\tau}_2\leq\frac{\beta^{-1}t}{2} \right]\\
\leq &\,e^{-\frac{\beta C_\nabla t }{8} } +\mathbb{P}_{\bar{y}_1} \left[\tilde{\tau}_1\leq \frac{\beta^{-1}t}{2}\right]\\&\,+\mathbb{P}_{\bar{y}_2} \left[\tilde{\tau}_2\leq \frac{\beta^{-1}t}{2} \right]\, ,
\end{align}
where $\bar{y}_1\in \partial B_{\gamma^{-1}\sqrt{N}\delta}$ and $\bar{y}_2 \in \partial B_{\gamma\sqrt{N}\delta}$ are arbitrary, and the last inequality can be obtained using a similar argument to the one used to obtain~\eqref{eq:stoppingbound}. 
The right hand side of the above expression goes to $0$ as $\beta \to \infty$, by the dominated convergence theorem. 

We now treat the third term in the bracket on the right hand side of~\eqref{eq:termstobeest} in a similar manner. Let $\hat{\tau}$ be the first hitting time of $B_{\kappa\sqrt{N}\delta}$ for some $1<\kappa<\gamma^{-1}$ and set $C_{V_0}:=\inf_{y\in B_{\kappa\sqrt{N}\delta}^c}V_0(y)$. We then have the following estimate
\begin{align}
&\,\esup_{y \in B_{\gamma^{-1} \sqrt{N}\delta}^c}\E_y \left[e^{-\beta^2 \int_0^{\beta^{-1}t}V_0(B_s) \, \mathrm{d}s} \right] \\
\leq & \,\esup_{y \in B_{\gamma^{-1} \sqrt{N}\delta}^c}\E_y \left[e^{-\frac{\beta  C_{V_0}t}{2}  } \right] +\esup_{y \in B_{\gamma^{-1} \sqrt{N}\delta}^c} \mathbb{P}_y\left[\hat{\tau}\leq\frac{\beta^{-1}t}{2}\right]\, ,
\end{align}
where the first term converges to $0$ as $\beta \to \infty$. 
The second term also converges to $0$ as $\beta \to \infty$ using similar arguments as for the other two terms on the right hand side of~\eqref{eq:termstobeest}. 
This completes the proof.
\end{proof}

\begin{lemma}
\label{lemma:Resolvent1}
For $f\in L^2(\R^{(N-1)d})$ and $\lambda>0$, it holds that 
\begin{equation}
\lim_{\beta \to \infty }\sup_{\norm*{f}_{L^2(\R^{(N-1)d})}\leq 1}\norm*{\overline{R(\lambda,-S_\beta)(f\rvert)}-R(\lambda,-\tilde{S}_\beta)f}_{L^2(\R^{(N-1)d})}=0 
\end{equation}
\begin{proof}
Fixing $\lambda>0$, by \cite[II.1.10]{EN00},~we write
\[
R(\lambda,\tilde{S}_\beta)f\,=\,\int_0^\infty e^{-\lambda t}e^{-\tilde{S}_\beta t}f\,\mathrm{d}t,\qquad \forall  f\in L^2(\R^{(N-1)d})\,, 
\]
and observe that 
\[
\overline{R(\lambda,-S_\beta)(f\rvert)}\,=\,\int_0^\infty e^{-\lambda t}\overline{e^{-S_\beta t}(f\rvert)}\,\mathrm{d}t, \qquad \forall f\in L^2(\R^{(N-1)d})\,. 
\]
Using~\cref{lemma:SemigroupConvergence}, the fact that the heat semigroup is non-expansive in $L^2$, and the dominated convergence theorem, we have that
\begin{equation}
\begin{aligned}
&\lim_{\beta \to \infty }\sup_{\norm*{f}_{L^2(\R^{(N-1)d})}\leq 1}\norm*{\overline{R(\lambda,-S_\beta)(f\rvert)}-R(\lambda,-\tilde{S}_\beta)f}_{L^2(\R^{(N-1)d})}\\
&\qquad\le\,\lim_{\beta \to \infty }\sup_{\norm*{f}_{L^2(\R^{(N-1)d})}\leq 1} \,\int_0^\infty e^{-\lambda t}\norm*{\overline{e^{-S_\beta t}(f\rvert)} - e^{-\tilde{S}_\beta t}f}_{L^2(\R^{(N-1)d})} \,\mathrm{d}t =0 \,  .
\end{aligned}
\end{equation}
Note, to apply the dominated convergence theorem, we have used the fact that both the semigroups are non-expansive since $\tilde{S}_\beta,S_\beta$ are non-negative.
\end{proof}
\end{lemma}

We have therefore shown that $R(\lambda,-\tilde{S}_\beta)$ converges in a generalized norm sense to $R(\lambda,-S_\beta)$ for $\lambda\in\R$, $\lambda>0$ as $\beta \to \infty$. We now extend this convergence to all $\lambda \in \Sigma \coloneqq \C\setminus(-\infty,0]$.

\begin{prop}
\label{prop:ResolventConvergence}
For any $f \in C_0(\R;\C)$\footnote{The space of continuous functions which vanish at infinity equipped with the supremum norm.}, we have that
\begin{equation}
\lim_{\beta \to \infty} \lVert f(S_\beta)-f(\tilde{S}_\beta) \rVert =0 \,.
\label{eq:functionconvergence}
\end{equation}
It follows from this and~\cref{lemma:Resolvent1} that for any $\lambda\in\Sigma$, 
\begin{equation}
\label{eq:normresolvent}
\lim_{\beta \to \infty }\left|\norm*{R(\lambda,-S_\beta)}_{L^2(B_{\sqrt{N}\delta})}-\norm*{R(\lambda,-\tilde{S}_\beta)}_{L^2(\R^{(N-1)d})}\right| =0 \,. 
\end{equation}
\end{prop}
\begin{proof}
We begin by demonstrating the convergence in \eqref{eq:normresolvent}~when $\lambda>0$. To this end, note that for $\lambda\notin\spec (-S_\beta)$ we have that 
\begin{equation}
\begin{aligned}
\label{eq:Resolvent}
\norm*{\overline{R(\lambda,-S_\beta)(\cdot\rvert)}}_{L^2(\R^{(N-1)d})}\,&=\,\sup_{\norm*{f}_{L^2(\R^{(N-1)d})}=1}\norm*{\overline{R(\lambda,-S_\beta)(f\rvert)}}_{L^2(\R^{(N-1)d})} \\
&=\,\sup_{\norm*{f}_{L^2(B_{\sqrt{N}\delta})}=1}\norm*{R(\lambda,-S_\beta)(f)}_{L^2(B_{\sqrt{N}\delta})}\\
&=\,\norm*{R(\lambda,-S_\beta)}_{L^2(B_{\sqrt{N}\delta})}\,. 
\end{aligned}
\end{equation}
By~\cref{lemma:Resolvent1}, for $\lambda>0$ we then have that 
\begin{equation}
\label{eq:Resolvent2}
\begin{aligned}
&\abs*{\norm*{R(\lambda,-S_\beta)}_{L^2(B_{\sqrt{N}\delta})} - \norm*{R(\lambda,-\tilde{S}_\beta)}_{L^2(\R^{(N-1)d})}}\\
&\qquad\le\, \norm*{\overline{R(\lambda,-S_\beta)(\cdot\rvert)} - R(\lambda,-\tilde{S}_\beta)}_{L^2(\R^{(N-1)d})} \to 0 \, ,
\end{aligned}
\end{equation}
as $\beta \to \infty$. 

The rest of the proof relies on the following two claims.\\

\noindent {\bf Claim 1.}~
Note that since the operators $S_\beta, \tilde{S}_\beta$ are self-adjoint operators $f(S_\beta),f(\tilde{S}_\beta)$ are well-defined as bounded operators on $L^2(\bar{B}_{\sqrt{N}\delta})$ and $L^2(\R^{(N-1)d})$, respectively, for all $f \in C_0(\R;\C)$ (see~\cite[Theorem 5.9]{S12}). Then, the set $F$, defined as 
\[
F\,\coloneqq\,
\bigg\{f\in C_0(\R; \C)\,:\,
\lim_{\beta \rightarrow\infty}\norm*{f(\tilde{S}_\beta)-\overline{f(S_\beta)(\cdot\rvert)}}_{L^2(\R^{(N-1)d})}=0\bigg\}\,, 
\]
is a $*$-subalgebra of $C_0(\R;\C)$, i.e.~it is closed under addition, involution, and multiplication. The closure under addition and involution can be checked in straightforward manner. To show that it is closed under multiplication, we pick some $f,g\in F$ and argue as follows
\[
\begin{aligned}
&\,\norm*{\overline{fg(S_\beta)(\cdot\rvert)}-(fg)(\tilde{S}_\beta)}_{L^2(\R^{(N-1)d})}\\
\le & \,\norm*{g(S_\beta)}_{L^2(\bar{B}_{\sqrt{N}\delta})}\norm*{\overline{f(S_\beta)(\cdot\rvert)}-f(\tilde{S}_\beta)}_{L^2(\R^{(N-1)d})}\\
&\,+\norm*{f(\tilde{S}_\beta)}_{L^2(\R^{(N-1)d})}\norm*{\overline{g(S_\beta)(\cdot\rvert)}-g(\tilde{S}_\beta)}_{L^2(\R^{(N-1)d})} \to 0
\end{aligned}
\]
as $\beta \to \infty$, since the operator norms of $g(S_\beta),\,f(\tilde{S}_\beta)$, are bounded by the uniform norms of $f,g$.
 \\

\noindent {\bf Claim 2.} We now assert that $F=C_0(\R;\C)$. 
We first note that $f_\lambda(\cdot)\coloneqq(\lambda-\cdot)^{-1}$ is contained in $F$ for each $\lambda>0$, by~\cref{lemma:Resolvent1}.  
Thus, $F$ vanishes nowhere and separates points, from which it follows that $F$ is dense in $C_0(\R;\C)$ by the Stone--Weierstrass theorem\footnote{We apply the standard Stone--Weierstrass theorem to the one-point compactification}. 
An $\eps/3$-argument shows that $F$ is closed in $C_0(\R;\C)$, thus proving the claim.
\\

We now conclude the proof of \cref{prop:ResolventConvergence}.
We have already shown~\eqref{eq:Resolvent} for $\lambda>0$. 
We are now left to show it for all $\lambda \in \C \setminus \R$ . 
For any such $\lambda$, the function $f_\lambda(\cdot)=(\lambda-\cdot)^{-1}$ is an element of $C_0(\R;\Sigma)$, and consequently of $F(\Sigma)$. 
The result thus follows. 
\end{proof}

\begin{prop}
\label{thm:SpectralConvergence} 
Both $\tilde{S}_\beta,S_\beta$ have pure point spectrum for all $\beta>0$ with a simple first eigenvalue. Furthermore, if $0\leq \tilde{\lambda}_{1,\beta}<\tilde{\lambda}_{2,\beta}\leq \tilde{\lambda}_{3,\beta}\leq   \dots$ (resp. $0<{\lambda}_{1,\beta}\leq{\lambda}_{2,\beta}\leq {\lambda}_{3,\beta} \leq \dots$) denote the eigenvalues of $\tilde{S}_\beta$ (resp. $S_\beta$), then for all $k \geq 1$ and $\eps>0$, there exists a $\bar{\beta}>0$ such that for all $\beta \geq \bar{\beta}$,  the set
\begin{equation}
\Lambda(k,\beta)\coloneqq \{\lambda\in \spec(S_\beta) :|\tilde{\lambda}_{k,\beta}-\lambda|\leq \eps\}\,,
\label{eq:spec1}
\end{equation}
is non-empty. 
Analogously, for all $k \geq 1$ and $\eps>0$, there exists a $\bar{\beta}>0$ such that for all $\beta \geq \bar{\beta}$, the set
\begin{equation}
\tilde{\Lambda}(k,\beta)\coloneqq\{\tilde{\lambda} \in \spec(\tilde{S}_\beta):|\lambda_{k,\beta}-\tilde{\lambda}|\leq \eps\}\,,
\label{eq:spec2}
\end{equation}
is non-empty.
\end{prop}
\begin{proof}
We already know that $S_\beta$ has pure point spectrum with a simple first eigenvalue since its spectrum is the same as that of $-L_D$\footnote{Note that here we write $\lambda_{k,\beta,N}=\lambda_{k,\beta}$, emphasising the dependence on $\beta$ and suppressing the dependence on $N$}. For $\tilde{S}_\beta$, the discreteness of the spectrum follows from~\cite[Theorem 1]{Sim09}  and the simplicity of the first eigenvalue follows from the same arguments as in the proof of \cite[Proposition 2.1]{AB97}, noting that $\tilde{S}_\beta$ is essentially self-adjoint in $L^2(\R^{(N-1)d})$. For $\lambda\in \C\setminus \R$, $\norm*{R(-\lambda^*,-S_\beta)}_{L^2(\bar{B}_{\sqrt{N}\delta})}$
is equal to the reciprocal of the distance of $\lambda$ from the spectrum of $S$ (and similarly for $R(-\lambda^*,-\tilde{S}_\beta)$).  Hence, for $\lambda\in\spec(S_\beta)$, using~\cref{prop:ResolventConvergence} and the fact that $\spec (S)\subseteq \R$, observe that 
\begin{align}
\lim_{\beta \to \infty}\mathrm{dist}(\lambda+i,\spec(\tilde{S}_\beta))^{-1}\,&=\,
\lim_{\beta\to \infty}\norm*{R(-\lambda+i,-\tilde{S}_\beta)}_{L^2(\R^{(N-1)d})}\\
&=\,1, 
\end{align}
from which~\eqref{eq:spec1} follows after applying the triangle inequality. The proof of~\eqref{eq:spec2} is similar.
\end{proof}

\subsection{\texorpdfstring{$\beta$}{beta}-asymptotics}
We are now prepared to apply the results of~\cite{Simon1983}. 
We rewrite the operator $\tilde{S}_\beta$ in the following form,
\begin{equation}
\label{eq:Sn}
\tilde{S}_\beta\,=\,-\beta^{-1}\Delta + \beta h+ g , 
\end{equation}
where $h,g$ can be determined from $V_\beta,V_0$. Note that we will check that $h,g$ satisfy the assumptions of~\cite[Theorem 1.1]{Simon1983} which we write out below:
\begin{enumerate}
\item[(A1)] the functions $h,\,g,$ are smooth, 
\item[(A2)] $g$ is bounded from below, and $h\ge0$, 
\item[(A3)] $h$ has a unique zero at zero, and $h(y)>0$ for $x\neq0$, 
\item[(A4)] the matrix $A=\frac12(D^2h)(0)$ is strictly positive-definite. 
\end{enumerate} 
It is clear that assumptions (A1)-(A3) are satisfied. Before we verify (A4), we introduce some relevant quantitites. Let $\{a_i^2\}_{i=1}^{(N-1)d}$ denote the eigenvalues of $A$ in ascending order, and for $\underline{n}\in\N_0^{(N-1)d}$ define 
\begin{equation}
\label{eq:E_undern}
E_{\underline{n}}^N\,\coloneqq\,g(0) + \sum_{i=1}^{(N-1)d}(2n_i+1)a_i\,. 
\end{equation}
We now relabel the $E_{\underline{n}}^N$ by arranging them in ascending order as $E_1^N \leq E_2^N \leq E_3^N \cdots$.  We then have the following result.
\begin{prop}
The matrix $A$ is strictly positive-definite and all its eigenvalues are exactly equal to $(w''(0))^2$. It follows from this that $0=E_1^N<E_2^N=w''(0)$.
\end{prop}
\begin{proof}
Notice that
\begin{equation}
A\,=\,\frac12 (D^2 h)(0)\,=\, \frac14(D^2U_N)^2 (0)\, .
\end{equation}
It follows that the eigenvalues of $A$ can be obtained from the squares of the eigenvalues of $(D^2U_N)(0)$. Since $(D^2U_N)(0)$ is just the restriction of $(D^2H_N)(0)$ to $\Gamma_N^\perp$ and $\Gamma_N$ is exactly the eigenspace of the eigenvalue $0$, it follows that the eigenvalues of $(D^2U_N)(0)$ are the same as those of $(D^2H_N)(0)$ excluding $0$. 

One can now check that $(D^2H_N)(0)$ is a block circulant matrix with diagonal blocks given by
\begin{equation}
\frac{N-1}{N}w''(0)\mathrm{Id}
\end{equation}
and off-diagonal blocks given by 
\begin{equation}
-\frac{1}{N}w''(0)\mathrm{Id} \, .
\end{equation}
Using the properties of circulant matrices, one can check that its spectrum consists of the eigenvalue $0$ with multiplicity $d$ and the eigenvalue $w''(0)$ having multiplicity $(N-1)d$. It follows from this and the previous discussion that $a_i= \frac{w''(0)}{2}$ for all $i=1,\dots,(N-1)d$.  Note further that $g(0)=-\frac12\Delta U_N(0)$. We thus have that
\begin{equation}
E_1^N\,=\, -\frac12\Delta U_N(0) +\frac{w''(0)(N-1)d}{2} =0\,,
\end{equation}
where we have used the fact that the $\Delta U_N(0)$ is just the sum of the eigenvalues of $(D^2U_N)(0)$. By similar arguments, we have that
\begin{equation}
E_2^N \,=\, w''(0)\,.
\end{equation}
This completes the proof.
\end{proof}
We now complete the proof of~\cref{thm:simon}.
\begin{prop}
For all $k \geq1$, we have that
\begin{equation}
\lim_{\beta \to \infty}\tilde{\lambda}_{k,\beta}\,=\,E_k^N \,.
\label{eq:Simon1}
\end{equation}
Furthermore, for $k=1,2$, we have that
\begin{equation}
\lim_{\beta \to \infty}\lambda_{k,\beta}\,=\,E_k^N\, .
\label{eq:Simon2}
\end{equation}
\end{prop}
\begin{proof}
The proof of~\eqref{eq:Simon1}~is a direct consequence of \cite[Theorem 1.1]{S82}. We proceed to the proof of~\eqref{eq:Simon2}. The proof in the $k=1$ case follows directly from~\cref{thm:mathieu} (clearly $E_1^N=0$).

We now proceed to the $k=2$ case.  Note that there exists some $\delta>0$ such that the $E_1^N,E_2^N$ are separated by at least $\delta>0$. 
Then, for any $\eps\in (0,\delta/6)$, we can use~\cref{prop:ResolventConvergence}, \cref{thm:mathieu}, and the first equality in~\eqref{eq:Simon1}, to argue that there exists a $\bar{\beta}<\infty$ such that for all $\beta \geq \bar{\beta}$, 
\begin{align}
|\lambda_{k,\beta}-\tilde{\lambda}|\,\leq \, \eps, \quad &|\tilde{\lambda}_{k,\beta}-\lambda|\,\leq \,\eps , \quad \forall (\lambda,\tilde{\lambda}) \in \Lambda(k,\beta)\times \tilde{\Lambda}(k,\beta), \, k=1,2\label{eq:lest1}\\
&\left|\tilde{\lambda}_{2,\beta}-E_2^N\right|\leq \,\eps, 
\label{eq:lest2}\\
|\lambda_{1,\beta}|\leq \eps \quad&\,|\tilde{\lambda}_{1,\beta}|\leq \eps\, .
\label{eq:improvedtilde}
\end{align}

We start by showing that for $\beta\geq \bar{\beta}$ $\lambda_{m,\beta}\notin \Lambda(1,\beta)$ for all $m\geq 2$. We note by~\cite[Theorem 5.9]{S12} that for all bounded, measurable $f:\R \to \C$, the operators $f(S_\beta),f(\tilde{S}_\beta)$ are well-defined as bounded operators and can be expressed as
\begin{align}
\int_{\bar{B}_{\sqrt{N}\delta}}\psi f(S_\beta)\varphi \, \dd x \coloneqq &\, \int_{\R} f(\lambda) \, \dd\mu_{\beta,\psi,\varphi}(\lambda) \\
\int_{\R^{(N-1)d}}\tilde{\psi} f(\tilde{S}_\beta)\tilde{\varphi} \, \dd x \coloneqq &\, \int_{\R} f(\lambda) \, \dd\tilde{\mu}_{\beta,\tilde{\psi},\tilde{\varphi}}(\lambda) 
\end{align} 
for $\psi,\varphi\in L^2(\bar{B}_{\sqrt{N}\delta})\,\tilde{\psi},\tilde{\varphi}\in L^2(\R^{(N-1)d})$, where $\mu_{\beta,\psi,\varphi},\tilde{\mu}_{\beta,\tilde{\psi},\tilde{\varphi}}$ are complex Borel measures supported on the spectrum of $S_\beta$ and $\tilde{S}_\beta$, respectively. We now consider the operators $\mathbf{1}_{A_\eps}(S_\beta),\mathbf{1}_{A_\eps}(\tilde{S}_\beta)$, where $A_\eps=[-3\eps,3\eps]$  and show that they converge to each other in generalised norm sense. To this end, we consider continuous and compactly supported functions $f \leq \mathbf{1}_{A_\eps}\leq g$ such that $\mathrm{supp}\,(g-f) \subset (-4\eps, -2\eps)\bigcup(-2\eps, 4\eps)$. We then have that
\allowdisplaybreaks
\begin{align}
&\, \lVert \overline{\mathbf{1}_{A_\eps}(S_\beta)}(\cdot\rvert)-\mathbf{1}_{A_\eps}(\tilde{S}_\beta)(\cdot) \rVert_{L^2(\R^{(N-1)d})}\\
=& \, \sup_{\norm{\psi}_{L^2(\R^{(N-1)d})},\norm{\psi}_{L^2(\R^{(N-1)d})}\leq 1}\int_{\bar{B}_{\sqrt{N}\delta}} \psi \rvert\times\overline{\mathbf{1}_{A_\eps}(S_\beta)\varphi\rvert} \dd x -\int_{\bar{B}_{\sqrt{N}\delta}} \psi \mathbf{1}_{A_\eps}(\tilde{S}_\beta)\varphi \dd x\\
=&\, \sup_{\norm{\psi}_{L^2(\R^{(N-1)d})},\norm{\psi}_{L^2(\R^{(N-1)d})}\leq 1} \int_{\R} \mathbf{1}_{A_\eps}(\lambda) \, \dd\mu_{\beta,\psi \rvert,\varphi\rvert}(\lambda) - \int_{\R} \mathbf{1}_{A_\eps}(\lambda) \, \dd\tilde{\mu}_{\beta,\psi ,\varphi}(\lambda) \\
\leq & \, \sup_{\norm{\psi}_{L^2(\R^{(N-1)d})},\norm{\psi}_{L^2(\R^{(N-1)d})}\leq 1} \int_{\R} g(\lambda) \, \dd\mu_{\beta,\psi \rvert,\varphi\rvert}(\lambda) - \int_{\R} f(\lambda) \, \dd\tilde{\mu}_{\beta,\psi ,\varphi}(\lambda) \\
=&\, \sup_{\norm{\psi}_{L^2(\R^{(N-1)d})},\norm{\psi}_{L^2(\R^{(N-1)d})}\leq 1} \int_{\R} g(\lambda) \, \dd\mu_{\beta,\psi \rvert,\varphi\rvert}(\lambda)- \int_{\R} g(\lambda) \, \dd\tilde{\mu}_{\beta,\psi ,\varphi}(\lambda)  \\&\,+ \int_{\R} (g-f)(\lambda) \, \dd\tilde{\mu}_{\beta,\psi ,\varphi}(\lambda) \\
\leq &\, \lVert \overline{g(S_\beta)}(\cdot\rvert)-g(\tilde{S}_\beta)(\cdot) \rVert_{L^2(\R^{(N-1)d})}\\ &\,+ \sup_{\norm{\psi}_{L^2(\R^{(N-1)d})},\norm{\psi}_{L^2(\R^{(N-1)d})}\leq 1} \int_{\R} (g-f)(\lambda) \, \dd\tilde{\mu}_{\beta,\psi ,\varphi}(\lambda)\, .\label{eq:opnormconv}
\end{align}
Note that using~\cref{prop:ResolventConvergence}, the first term on the right hand side of~\eqref{eq:opnormconv} goes to $0$ as $\beta \to \infty$. For the second term, we note that by~\eqref{eq:lest1},~\eqref{eq:lest2}, and~\eqref{eq:improvedtilde}, for $\beta \geq \bar{\beta}$ $\spec (\tilde{S}_\beta)\cap \mathrm{supp}\, (g-f)= \emptyset$. Thus, by~\cite[Proposition 5.10]{S12}, the second term on the right hand side of~\eqref{eq:opnormconv} is $0$.

We now note that $\tilde{P}_{\beta,\eps} \coloneqq \mathbf{1}_{A_\eps}(\tilde{S}_\beta)$ (resp. $P_{\beta,\eps} \coloneqq \mathbf{1}_{A_\eps}(S_\beta)$) is nothing but the orthogonal projection onto the part of the spectrum of $\tilde{S}_\beta$ (resp. $S_\beta$) contained in $A_\eps$ -- see for instance \cite[Problem 2, Chapter VII]{RS1}. For $\beta \geq \bar{\beta}$ by~\eqref{eq:lest1} and ~\eqref{eq:lest2}, $A_\eps \bigcap \spec (\tilde{S}_\beta)=\{\tilde{\lambda}_{1,\beta}\}$ which has multiplicity $1$ and so $\tilde{P}_{\beta,\eps}$ has rank $1$. It follows from the operator norm convergence established in~\eqref{eq:opnormconv} that $P_{\beta,\eps}$ must also have rank $1$, for $\beta$ sufficiently large. Note that we can find a large $\bar{\beta}$\footnote{we abuse notation here by reusing $\bar{\beta}$} such that, for all $\beta\geq \bar{\beta}$,
\begin{equation}
 \lVert \overline{P_{\beta,\eps}}(\cdot\rvert)-\tilde{P}_{\beta,\eps}(\cdot) \rVert_{L^2(\R^{(N-1)d})} <1\,.
\end{equation}
This must imply that $\mathrm{rank}(P_{\beta,\eps})\leq\mathrm{rank}(\tilde{P}_{\beta,\eps})$. Assume by contradiction that this is not the case. Then, there must exist a non-trivial unit norm element $\varphi \in L^2(\bar{B}_{\sqrt{N}\delta})\subset L^2(\R^{(N-1)d})$ such that $\varphi \in \mathrm{ker}(\tilde{P}_{\beta,\eps})\bigcap \mathrm{im}(P_{\beta,\eps})$ (if not the map $P_{\beta,\eps}:\mathrm{im}(P_{\beta,\eps}) \to \mathrm{im}(\tilde{P}_{\beta,\eps})$ is an injection and so $\mathrm{rank}(P_{\beta,\eps})\leq\mathrm{rank}(\tilde{P}_{\beta,\eps})$.) The existence of such an element would however imply that 
\begin{align}
 &\,\lVert \overline{P_{\beta,\eps}}(\cdot\rvert)-\tilde{P}_{\beta,\eps}(\cdot) \rVert_{L^2(\R^{(N-1)d})}\geq 1\,, 
\end{align}
which is a contradiction. Thus, for $\beta \geq \bar{\beta}$, $\mathrm{rank}(P_{\beta,\eps})\leq 1$. But, we already know from~\eqref{eq:improvedtilde} that for all $\beta\geq \bar{\beta}$ $|\lambda_{1,\beta}|\leq \eps$ and so  $\mathrm{rank}(P_{\beta,\eps})=1$. As a consequence, we have that the set $A_\eps$ cannot contain anymore eigenvalues of $\tilde{S}_\beta$ and in fact
\begin{equation}
|\lambda_{m,\beta}|>3 \eps\, ,
\label{eq:separation}
\end{equation}
for all $m\geq 2$ from which it must follow that $\lambda_{m,\beta}\notin \Lambda(1,\beta)$ for all $m\geq 2$. 

We now argue that 
\begin{equation}
|\lambda_{2,\beta}-E_2^N|\leq 3 \eps
\label{eq:final}
\end{equation}
 for all $\beta \geq \bar{\beta}$. Assume by contradiction that this is not the case, i.e.~for some $\beta \geq \bar{\beta}$ we have that $|\lambda_{2,\beta}-E_2^N|> 3 \eps$.  We know that $\Lambda(2,\beta)$ contains $\lambda_{m',\beta}$ for some $m'\geq 2$ (the case $m'= 1$ can be ruled out by~\eqref{eq:lest2} and~\eqref{eq:improvedtilde}). Note also that $m'\neq 2$. Indeed, if $m'=2$ then we would have
\begin{align}
&\,\left|\lambda_{2,\beta}-E_2^N\right|\\
\leq &\left|\lambda_{2,\beta}-\tilde{\lambda}_{2,\beta}\right|+\left|\tilde{\lambda}_{2,\beta}-E_2^N\right| \leq 2\eps\, ,
\end{align}
which is false by assumption. Thus, $m'>2$. This, implies that
\begin{equation}
 \left|\lambda_{m',\beta}-E_2^N\right| \leq 2\eps \, ,
\end{equation}
which would imply that $\lambda_{m',\beta}<\lambda_{2,\beta}$ which is clearly a contradiction. Thus,~\eqref{eq:final} holds true. Since $\eps>0$ is arbitrary, this establishes~\eqref{eq:Simon2} for $k=2$.
\end{proof}

\section{Proof of \texorpdfstring{\cref{thm:multiscale}}{multiscale convergence result}}

\begin{proof}[Proof of~\cref{thm:multiscale}]

i) For any $f \in C(B_{\star,\sqrt{N}\delta})$, we have
\begin{align}
\E\left[f(\hat{Y}_t)|Y_0=y\right]=& \,  \E\left[f(\hat{Y}_t)\mathbf{1}_{t<\tau_\delta}|Y_0=y\right] + \E\left[f(\hat{Y}_t)\mathbf{1}_{t\geq \tau_\delta}|Y_0=y\right]\\
=&\,  (P_t f)(y) +f(\star)\E \left[\mathbf{1}_{t\geq \tau_\delta}\right]\\
=& \, (P_t f)(y) +f(\star) -  f(\star) (P_t 1)(y)  \,, 
\end{align}
where we have simply used the definition of $P_t$ introduced in \cref{sub:semigroups_and_generators} and the fact that $f(\hat{Y}_{\tau_\delta})=f(\star)$. Using the heat kernel representation introduced in \cref{sub:semigroups_and_generators}, we have
\begin{align}
 \E\left[f(\hat{Y}_t)|Y_0=y\right]=& \, e^{-\lambda_{1,\beta,N}t}\left(\int_{\bar{B}_{\sqrt{N}\delta}}f e_1 \,\mathrm{d}p_N \right)e_1(y) + (P_t \mathsf{Q}_{\geq 2}f)(y)\\
 &\, - f(\star) (P_t \mathsf{Q}_{\geq 2}1)(y) \label{eq:preidentity}\\
 &\, + f(\star)\left(1- e^{-\lambda_{1,\beta,N}t}\left(\int_{\bar{B}_{\sqrt{N}\delta}} e_1 \,\mathrm{d}p_N \right)e_1(y) \right)\, .
\end{align} 
Using the explicit expression for the QSD introduced in \cref{sub:the_quasi_ergodic_and_quasi_stationary_distributions}, we can rewrite the above expression as 
\begin{align}
\E\left[f(\hat{Y}_t)|Y_0=y\right] =& \, \gamma_{\beta,N}(y) e^{-\lambda_{1,\beta,N}t}\left(\int_{\bar{B}_{\sqrt{N}\delta}}f \,\mathrm{d}q_N \right) + (1-\gamma_{\beta,N}(y) e^{-\lambda_{1,\beta,N}t}) f(\star)\\
& \, +   (P_t \mathsf{Q}_{\geq 2}(f-f(\star)))(y)\, ,
\end{align}
where $\gamma_{\beta,N}(y)$ is defined as 
\begin{equation}
 \gamma_{\beta,N}(y):=e_1(y)\left(\int_{\bar{B}_{\sqrt{N}\delta}}e_1 \, \mathrm{d}p_N \right) \, .
\end{equation} 
The identity~\eqref{eq:multiscaleidentity} then follows from~\eqref{eq:preidentity} and the fact that
\[
\E\left[f(\hat{Y}_t)|Y_0 \sim \nu_N\right]=\int_{\bar{B}_{\sqrt{N}\delta}} \E\left[f(\hat{Y}_t)|Y_0=y\right]\, \mathrm{d}\nu_N(y) \, .
\]
\medskip

ii) To prove the statements in the theorem, we first control the right hand side of~\eqref{eq:multiscaleidentity}. We now assume that $\nu_N$ is absolutely continuous with respect to $p_N$. We then have the following bound
\begin{align}
  &\int_{\bar{B}_{\sqrt{N}\delta}}(P_t \mathsf{Q}_{\geq 2}(f-f(\star)))\, \mathrm{d}\nu_N\\
  =& \, \int_{\bar{B}_{\sqrt{N}\delta}}(f-f(\star))P_t \mathsf{Q}_{\geq 2} \frac{\mathrm{d}\nu_N}{\mathrm{d}p_N}\, \mathrm{d}p_N \\
  \leq & e^{-\lambda_{2,\beta,N}t}\lVert f-f(\star) \rVert_{L^2(\bar{B}_{\sqrt{N}\delta};p_N)} \left \lVert \frac{\mathrm{d}\nu_N}{\mathrm{d}p_N} \right \rVert_{L^{2}(\bar{B}_{\sqrt{N}\delta};p_N)}\, 
\end{align} 
where in the last step we have applied the Cauchy--Schwartz inequality and have used the heat kernel representation in~\eqref{eq:heatkernel} along with the definition of the projection $\mathsf{Q}_{\geq 2}$ in~\eqref{eq:projection}. If we now take the a supremum over all $f \in C(B_{\star,\sqrt{N}\delta})$ with $|f|\leq 1$, we can apply the Riesz--Markov--Kakutani representation theorem (since $B_{\star,\sqrt{N}\delta}$ is a compact Hausdorff space) to argue that~\eqref{eq:TVbound} holds true. 

If we now choose $f\in \mathrm{Lip}_1(B_{\star,\sqrt{N}\delta})$, we have
\begin{align}
&\lVert f-f(\star) \rVert_{L^2(\bar{B}_{\sqrt{N}\delta};p_N)} \left \lVert \frac{\mathrm{d}\nu_N}{\mathrm{d}p_N} \right \rVert_{L^{2}(\bar{B}_{\sqrt{N}\delta};p_N)}\\
 \leq & \, e^{-\lambda_{2,\beta,N}t} \sup_{\bar{B}_{\sqrt{N}\delta}}d_\star(\cdot,\star) \left \lVert \frac{\mathrm{d}\nu_N}{\mathrm{d}p_N} \right \rVert_{L^{2}(\bar{B}_{\sqrt{N}\delta};p_N)} \\
 \leq & \, \sqrt{N}\delta e^{-\lambda_{2,\beta,N} t} \, ,
\end{align}
where we have used~\eqref{eq:heatkernel} again and $d_\star$ is as defined in~\eqref{eq:starmetric}. Using the dual formulation of the $W_1$ distance (see~\eqref{eq:W1}) we obtain~\eqref{eq:W1bound}.

\medskip

(iii) We now move on to the proof of~\eqref{eq:alphaconvergence}. 
To this end, we employ~\cite[Theorem 2.4]{DF78}, which tells us if the ODE
\begin{equation}
\dot{y}(t)=-\nabla U_N(y(t))\, ,\quad y(0)\in B_{\sqrt{N}\delta} 
\label{eq:ode}
\end{equation}
only has solutions which converge to $0$, then
\begin{equation}
e_1 \to 1 \, , \textrm{uniformly on compact subsets of } B_{\sqrt{N}\delta}.
\end{equation}
For any $0<\delta'<\delta$, we then have
\begin{align}
& \,\left|\int_{\bar{B}_{\sqrt{N}\delta}}(e_1-1) \, \mathrm{d}p_N\right| \\
\leq & \,  \left|\int_{\bar{B}_{\sqrt{N}\delta'}}(e_1 -1)\, \mathrm{d}p_N\right|  + \left|\int_{\bar{B}_{\sqrt{N}\delta'}\setminus \bar{B}_{\sqrt{N}\delta'}}(e_1-1) \, \mathrm{d}p_N\right|\\
\leq & \, \lVert e_1-1 \rVert_{L^\infty(\bar{B}_{\sqrt{N}\delta'})}p_N(\bar{B}_{\sqrt{N}\delta'}) + \left(\lVert e_1 \rVert_{L^2(\bar{B}_{\sqrt{N}\delta};p_N)}+1\right)p_N^{\frac12}(\bar{B}_{\sqrt{N}\delta}\setminus \bar{B}_{\sqrt{N}\delta'}) \\
=& \, \lVert e_1-1 \rVert_{L^\infty(\bar{B}_{\sqrt{N}\delta'})}p_N(\bar{B}_{\sqrt{N}\delta'}) + 2 p_N^{\frac12}(\bar{B}_{\sqrt{N}\delta}\setminus \bar{B}_{\sqrt{N}\delta'}) \, ,
\end{align}
where for the penultimate identity we have applied the Cauchy--Schwarz inequality for the second term. As $\beta \to \infty$, the first term on the right hand side goes to $0$  due to the convergence of $e_1$ while the second term goes to $0$ because $p_N$ converges weakly to $\delta_0$ due to the fact that $U_N$ has a unique minimum at $0$ (see~\cref{lemma:mincrti}). 
A similar argument can be applied to the term $\int_{\bar{B}_{\sqrt{N}\delta}}(e_1-1) \, \mathrm{d}\nu_N$ with the only difference being that we first choose $\delta'$ close to $\delta$ so that $\nu_N(\bar{B}_{\sqrt{N}\delta}\setminus \bar{B}_{\sqrt{N}\delta'})$ is arbitrarily small (this is possible since $\nu_N\ll p_N \ll \mathrm{d}y$) and then carry out the same argument as above.

Now, as shown in~\cref{lemma:mincrti} (see~\eqref{eq:1criticalpoint}), $-\langle\nabla U_N(y),y \rangle <0$ for all $y\neq 0$ and $\nabla U_N(0)=0$. Thus, the ODE~\eqref{eq:ode} only has solutions which converge to $0$. The result then follows.

\medskip

iv) Finally, we prove that the empirical measure of the droplet state, characterised by the QSD $q_N$, has a second moment which scales as $\beta^{-1}$. 
For $Y \sim q_N$, we compute 
\begin{align}
\mathbb{E}\left[\int_{\R^d}|x|^2 \, \mathrm{d}\mu^{(N),Y}\right]=&\, \frac{1}{N}\int_{\bar{B}_{\sqrt{N}\delta}} |y|^2 \, \mathrm{d}q_N \\
=& \, \frac{1}{N\displaystyle\int_{\bar{B}_{\sqrt{N}\delta}}e_1 \, \mathrm{d}p_N}\int_{\bar{B}_{\sqrt{N}\delta}} |y|^2 e_1 \, \mathrm{d}p_N \\
\leq & \, \frac{1}{N\displaystyle\int_{\bar{B}_{\sqrt{N}\delta}}e_1 \, \mathrm{d}p_N}\left(\int_{\bar{B}_{\sqrt{N}\delta}} |y|^4 \, \mathrm{d}p_N\right)^{\frac12} \, ,\label{eq:sizeestimatemiddle}
\end{align}
where we have applied the Cauchy--Schwarz inequality.  Let us now consider the integral in the numerator by itself:
\begin{align}
& \, \int_{\bar{B}_{\sqrt{N}\delta}} |y|^4 \, \mathrm{d}p_N\\
=&\, \frac{1}{\displaystyle\int_{\bar{B}_{\sqrt{N}\delta}}  e^{-\beta U_N(y)}\, \mathrm{d}y}\int_{\bar{B}_{\sqrt{N}\delta}}|y|^4  e^{-\beta U_N(y)}\, \mathrm{d}y \\
\leq & \,\frac{1}{\displaystyle\int_{\bar{B}_{\sqrt{N}\delta}}  e^{-\beta U_N(y)}\, \mathrm{d}y}\int_{\bar{B}_{\sqrt{N}\delta}}|y|^4  e^{-\beta c_w |y|^2}\, \mathrm{d}y \, ,
\end{align}
where in the last step we have used the result of~\cref{lemma:depth} and~\eqref{eq:quadraticlowerbound}. Rescaling using the map $\beta^{\frac12} y \mapsto y$, we have
\begin{align}
\, \int_{\bar{B}_{\sqrt{N}\delta}} |y|^4 \, \mathrm{d}p_N \leq \frac{1}{\beta^{2+\frac{d}{2}}\displaystyle\int_{\bar{B}_{\sqrt{N}\delta}}  e^{-\beta U_N(y)}\, \mathrm{d}y}\int_{\Gamma_N^\perp}|y|^4  e^{- c_w |y|^2}\, \mathrm{d}y \, .
\end{align}
Putting the above estimate together with~\eqref{eq:sizeestimatemiddle}, we obtain
\begin{equation}
\frac{\beta}{N}\int_{\bar{B}_{\sqrt{N}\delta}} |y|^2 \, \mathrm{d}q_N \leq \frac{1}{N\left(\displaystyle\int_{\bar{B}_{\sqrt{N}\delta}}e_1 \, \mathrm{d}p_N\right)\left(\beta^{\frac d2}\displaystyle\int_{\bar{B}_{\sqrt{N}\delta}}  e^{-\beta U_N(y)}\, \mathrm{d}y\right)^{\frac12}}\,.
\end{equation}
We have already shown that first integral in the denominator converges to $1$ as $\beta \to \infty$. For the second integral, we note that $\lim_{\beta \to \infty}\beta^{\frac d2}\displaystyle\int_{\bar{B}_{\sqrt{N}\delta}}  e^{-\beta U_N(y)}\, \mathrm{d}y  >0$ by a Laplace approximation argument and the fact that $U_N(0)=0$ is the unique minimum of $U_N$ in $\bar{B}_{\sqrt{N}\delta}$ and $D^2 U_N (0)>0$ (see~\cref{ass:W,lemma:mincrti}). Taking the limit superior, ~\eqref{eq:dropletsize} follows.
\end{proof}

\paragraph{\bf Acknowledgments}
ZPA and ME have been supported by Germany’s Excellence Strategy – The Berlin Mathematics Research Center MATH+ (EXC-2046/1, project ID:
390685689), via project EF45-5. ZPA additionally acknowledges support by the Saxony State Ministry of Science, and thanks both ScaDS.AI and the Max Planck Institute for Mathematics in the Sciences for hosting him during part of this research. 
Furthermore, ME thanks the DFG CRC 1114 and the Einstein Foundation for support. The research of RSG was partially supported by the DFG through the SPP 2410/1 ``Hyperbolic Balance Laws in Fluid Mechanics: Complexity, Scales, Randomness''.

\bibliographystyle{myalpha}
\bibliography{ref}

\end{document}